\documentclass[12pt]{amsart}
\usepackage{txfonts}      
\usepackage{amssymb}
\usepackage{eucal}
\usepackage{amsmath}
\usepackage{amscd}
\usepackage{xcolor}
\usepackage{multicol}
\usepackage[all]{xy}           
\usepackage{graphicx}
\usepackage{color}
\usepackage{colordvi}
\usepackage{xspace}
\usepackage{tikz}
\usepackage{makecell}
\usepackage{appendix}
\usepackage{amsthm}
\usepackage[misc]{ifsym}

\usepackage{ifpdf}
\ifpdf
\usepackage[colorlinks,final,backref=page,hyperindex]{hyperref}
\else
\usepackage[colorlinks,final,backref=page,hyperindex,hypertex]{hyperref}
\fi

\usepackage[active]{srcltx} 

\begin{document}


\newtheorem{thm}{Theorem}[section]
\newtheorem{lem}[thm]{Lemma}
\newtheorem{cor}[thm]{Corollary}
\newtheorem{pro}[thm]{Proposition}
\theoremstyle{definition}
\newtheorem{defi}[thm]{Definition}
\newtheorem{ex}[thm]{Example}
\newtheorem{rmk}[thm]{Remark}
\newtheorem{pdef}[thm]{Proposition-Definition}
\newtheorem{condition}[thm]{Condition}

\renewcommand{\labelenumi}{{\rm(\alph{enumi})}}
\renewcommand{\theenumi}{\alph{enumi}}

\baselineskip=14pt

\newcommand {\emptycomment}[1]{} 

\newcommand{\nc}{\newcommand}
\newcommand{\delete}[1]{}

\nc{\todo}[1]{\tred{To do:} #1}

\nc{\tred}[1]{\textcolor{red}{#1}}
\nc{\tblue}[1]{\textcolor{blue}{#1}}
\nc{\tgreen}[1]{\textcolor{green}{#1}}
\nc{\tpurple}[1]{\textcolor{purple}{#1}}
\nc{\tgray}[1]{\textcolor{gray}{#1}}
\nc{\torg}[1]{\textcolor{orange}{#1}}
\nc{\tmag}[1]{\textcolor{magenta}}
\nc{\btred}[1]{\textcolor{red}{\bf #1}}
\nc{\btblue}[1]{\textcolor{blue}{\bf #1}}
\nc{\btgreen}[1]{\textcolor{green}{\bf #1}}
\nc{\btpurple}[1]{\textcolor{purple}{\bf #1}}

	\nc{\mlabel}[1]{\label{#1}}  
	\nc{\mcite}[1]{\cite{#1}}  
	\nc{\mref}[1]{\ref{#1}}  
	\nc{\meqref}[1]{\eqref{#1}}  
	\nc{\mbibitem}[1]{\bibitem{#1}} 

\delete{
	\nc{\mlabel}[1]{\label{#1}  
		{ {\small\tgreen{\tt{{\ }(#1)}}}}}
	\nc{\mcite}[1]{\cite{#1}{\small{\tt{{\ }(#1)}}}}  
	\nc{\mref}[1]{\ref{#1}{\small{\tred{\tt{{\ }(#1)}}}}}  
	\nc{\meqref}[1]{\eqref{#1}{{\tt{{\ }(#1)}}}}  
	\nc{\mbibitem}[1]{\bibitem[\bf #1]{#1}} 
}

\nc{\zy}[1]{\textcolor{red}{Zhongyin:#1}}
\nc{\yy}[1]{\textcolor{blue}{Yanyong: #1}}
\nc{\li}[1]{\textcolor{purple}{#1}}
\nc{\lir}[1]{\textcolor{purple}{Li:#1}}


\nc{\tforall}{\ \ \text{for all }}
\nc{\hatot}{\,\widehat{\otimes} \,}
\nc{\complete}{completed\xspace}
\nc{\wdhat}[1]{\widehat{#1}}

\nc{\ts}{\mathfrak{p}}
\nc{\mts}{c_{(i)}\ot d_{(j)}}

\nc{\NA}{{\bf NA}}
\nc{\LA}{{\bf Lie}}
\nc{\CLA}{{\bf CLA}}

\nc{\cybe}{CYBE\xspace}
\nc{\nybe}{NYBE\xspace}
\nc{\ccybe}{CCYBE\xspace}

\nc{\ndend}{pre-Novikov\xspace}
\nc{\calb}{\mathcal{B}}
\nc{\rk}{\mathrm{r}}
\newcommand{\g}{\mathfrak g}
\newcommand{\h}{\mathfrak h}
\newcommand{\pf}{\noindent{$Proof$.}\ }
\newcommand{\frkg}{\mathfrak g}
\newcommand{\frkh}{\mathfrak h}
\newcommand{\Id}{\rm{Id}}
\newcommand{\gl}{\mathfrak {gl}}
\newcommand{\ad}{\mathrm{ad}}
\newcommand{\add}{\frka\frkd}
\newcommand{\frka}{\mathfrak a}
\newcommand{\frkb}{\mathfrak b}
\newcommand{\frkc}{\mathfrak c}
\newcommand{\frkd}{\mathfrak d}
\newcommand {\comment}[1]{{\marginpar{*}\scriptsize\textbf{Comments:} #1}}


\nc{\disp}[1]{\displaystyle{#1}}
\nc{\bin}[2]{ (_{\stackrel{\scs{#1}}{\scs{#2}}})}  
\nc{\binc}[2]{ \left (\!\! \begin{array}{c} \scs{#1}\\
    \scs{#2} \end{array}\!\! \right )}  
\nc{\bincc}[2]{  \left ( {\scs{#1} \atop
    \vspace{-.5cm}\scs{#2}} \right )}  
\nc{\ot}{\otimes}
\nc{\sot}{{\scriptstyle{\ot}}}
\nc{\otm}{\overline{\ot}}
\nc{\ola}[1]{\stackrel{#1}{\la}}

\nc{\scs}[1]{\scriptstyle{#1}} \nc{\mrm}[1]{{\rm #1}}

\nc{\dirlim}{\displaystyle{\lim_{\longrightarrow}}\,}
\nc{\invlim}{\displaystyle{\lim_{\longleftarrow}}\,}

\nc{\bfk}{{\bf k}} \nc{\bfone}{{\bf 1}}
\nc{\rpr}{\circ}
\nc{\dpr}{{\tiny\diamond}}
\nc{\rprpm}{{\rpr}}

\nc{\mmbox}[1]{\mbox{\ #1\ }} \nc{\ann}{\mrm{ann}}
\nc{\Aut}{\mrm{Aut}} \nc{\can}{\mrm{can}}
\nc{\twoalg}{{two-sided algebra}\xspace}
\nc{\colim}{\mrm{colim}}
\nc{\Cont}{\mrm{Cont}} \nc{\rchar}{\mrm{char}}
\nc{\cok}{\mrm{coker}} \nc{\dtf}{{R-{\rm tf}}} \nc{\dtor}{{R-{\rm
tor}}}
\renewcommand{\det}{\mrm{det}}
\nc{\depth}{{\mrm d}}
\nc{\End}{\mrm{End}} \nc{\Ext}{\mrm{Ext}}
\nc{\Fil}{\mrm{Fil}} \nc{\Frob}{\mrm{Frob}} \nc{\Gal}{\mrm{Gal}}
\nc{\GL}{\mrm{GL}} \nc{\Hom}{\mrm{Hom}} \nc{\hsr}{\mrm{H}}
\nc{\hpol}{\mrm{HP}}  \nc{\id}{\mrm{id}} \nc{\im}{\mrm{im}}

\nc{\incl}{\mrm{incl}} \nc{\length}{\mrm{length}}
\nc{\LR}{\mrm{LR}} \nc{\mchar}{\rm char} \nc{\NC}{\mrm{NC}}
\nc{\mpart}{\mrm{part}} \nc{\pl}{\mrm{PL}}
\nc{\ql}{{\QQ_\ell}} \nc{\qp}{{\QQ_p}}
\nc{\rank}{\mrm{rank}} \nc{\rba}{\rm{RBA }} \nc{\rbas}{\rm{RBAs }}
\nc{\rbpl}{\mrm{RBPL}}
\nc{\rbw}{\rm{RBW }} \nc{\rbws}{\rm{RBWs }} \nc{\rcot}{\mrm{cot}}
\nc{\rest}{\rm{controlled}\xspace}
\nc{\rdef}{\mrm{def}} \nc{\rdiv}{{\rm div}} \nc{\rtf}{{\rm tf}}
\nc{\rtor}{{\rm tor}} \nc{\res}{\mrm{res}} \nc{\SL}{\mrm{SL}}
\nc{\Spec}{\mrm{Spec}} \nc{\tor}{\mrm{tor}} \nc{\Tr}{\mrm{Tr}}
\nc{\mtr}{\mrm{sk}}

\nc{\ab}{\mathbf{Ab}} \nc{\Alg}{\mathbf{Alg}}

\nc{\BA}{{\mathbb A}} \nc{\CC}{{\mathbb C}} \nc{\DD}{{\mathbb D}}
\nc{\EE}{{\mathbb E}} \nc{\FF}{{\mathbb F}} \nc{\GG}{{\mathbb G}}
\nc{\HH}{{\mathbb H}} \nc{\LL}{{\mathbb L}} \nc{\NN}{{\mathbb N}}
\nc{\QQ}{{\mathbb Q}} \nc{\RR}{{\mathbb R}} \nc{\BS}{{\mathbb{S}}} \nc{\TT}{{\mathbb T}}
\nc{\VV}{{\mathbb V}} \nc{\ZZ}{{\mathbb Z}}


\nc{\calao}{{\mathcal A}} \nc{\cala}{{\mathcal A}}
\nc{\calc}{{\mathcal C}} \nc{\cald}{{\mathcal D}}
\nc{\cale}{{\mathcal E}} \nc{\calf}{{\mathcal F}}
\nc{\calfr}{{{\mathcal F}^{\,r}}} \nc{\calfo}{{\mathcal F}^0}
\nc{\calfro}{{\mathcal F}^{\,r,0}} \nc{\oF}{\overline{F}}
\nc{\calg}{{\mathcal G}} \nc{\calh}{{\mathcal H}}
\nc{\cali}{{\mathcal I}} \nc{\calj}{{\mathcal J}}
\nc{\call}{{\mathcal L}} \nc{\calm}{{\mathcal M}}
\nc{\caln}{{\mathcal N}} \nc{\calo}{{\mathcal O}}
\nc{\calp}{{\mathcal P}} \nc{\calq}{{\mathcal Q}} \nc{\calr}{{\mathcal R}}
\nc{\calt}{{\mathcal T}} \nc{\caltr}{{\mathcal T}^{\,r}}
\nc{\calu}{{\mathcal U}} \nc{\calv}{{\mathcal V}}
\nc{\calw}{{\mathcal W}} \nc{\calx}{{\mathcal X}}
\nc{\CA}{\mathcal{A}}

\nc{\fraka}{{\mathfrak a}} \nc{\frakB}{{\mathfrak B}}
\nc{\frakb}{{\mathfrak b}} \nc{\frakd}{{\mathfrak d}}
\nc{\oD}{\overline{D}}
\nc{\frakF}{{\mathfrak F}} \nc{\frakg}{{\mathfrak g}}
\nc{\frakm}{{\mathfrak m}} \nc{\frakM}{{\mathfrak M}}
\nc{\frakMo}{{\mathfrak M}^0} \nc{\frakp}{{\mathfrak p}}
\nc{\frakS}{{\mathfrak S}} \nc{\frakSo}{{\mathfrak S}^0}
\nc{\fraks}{{\mathfrak s}} \nc{\os}{\overline{\fraks}}
\nc{\frakT}{{\mathfrak T}}
\nc{\oT}{\overline{T}}
\nc{\frakX}{{\mathfrak X}} \nc{\frakXo}{{\mathfrak X}^0}
\nc{\frakx}{{\mathbf x}}
\nc{\frakTx}{\frakT}      
\nc{\frakTa}{\frakT^a}        
\nc{\frakTxo}{\frakTx^0}   
\nc{\caltao}{\calt^{a,0}}   
\nc{\ox}{\overline{\frakx}} \nc{\fraky}{{\mathfrak y}}
\nc{\frakz}{{\mathfrak z}} \nc{\oX}{\overline{X}}

\font\cyr=wncyr10


\title{One-dimensional central extensions and simplicities of a class of left-symmetric conformal algebras}

\author{Zhongyin Xu}
\address{School of Mathematics, Hangzhou Normal University,
Hangzhou, 311121, China}
\email{Xzy@stu.hznu.edu.cn}

\author{Yanyong Hong}
\address{School of Mathematics, Hangzhou Normal University,
Hangzhou, 311121, China}
\email{yyhong@hznu.edu.cn}

\subjclass[2020]{17A30, 17D25, 17B60}
\keywords{Left-symmetric conformal algebra, pre-Novikov algebra, pre-Gel'fand-Dorfman algebra, Novikov-Poisson algebra, central extension}

\begin{abstract}
In this paper, we introduce the definition of pre-Gel'fand-Dorfman algebra and present several constructions. Moreover, we show that a class of left-symmetric conformal algebras named quadratic left-symmetric conformal algebras are one to one correspondence with pre-Gel'fand-Dorfman algebras.
Then we investigate the simplicities and central extensions of quadratic left-symmetric conformal algebras by a one-dimensional centre from the point of view of pre-Gel'fand-Dorfman algebras. We show that under some conditions, central extensions of quadratic left-symmetric conformal algebras by a one-dimensional centre can be characterized by four bilinear forms on pre-Gel'fand-Dorfman algebras. Several methods to construct simple quadratic left-symmetric conformal algebras from pre-Gel'fand-Dorfman algebras are also given.
\end{abstract}
\footnote{Yanyong Hong is the corresponding author.}
\maketitle

\section{Introduction}
The notion of Lie conformal algebra was introduced V. Kac in \cite{K1} to give an axiomatic description of singular part of the operator product expansion of chiral fields in conformal field theory. Lie conformal algebras have close connections with vertex algebras \cite{K1},  infinite-dimensional Lie algebras satisfying the locality property \cite{K} and Hamiltonian formalism in the theory of nonlinear evolution equations \cite{BDK}. Structure theory and representation theory of finite Lie conformal algebras have been well developed (see \cite{DK, BKV, CK}).

Based on an equivalent characterization of vertex algebra by the notions of Lie conformal algebra and left-symmetric algebra given in \cite{BK}, the definition of left-symmetric conformal algebra was introduced in \cite{HL} to investigate whether there exist compatible left-symmetric algebra structures on formal distribution Lie algebras.  Notice that left-symmetric conformal algebras are a class of special left-symmetric pseudo-algebras  introduced in \cite{Wu}.
Similar to the classical case, the conformal commutator of a left-symmetric conformal algebra is a Lie conformal algebra.
Moreover, finite left-symmetric conformal algebras which are free $\mathbb{C}[\partial]$-modules can naturally  provide the solutions of conformal Yang-Baxter equation and conformal $S$-equation  \cite{HB}. There have been some works on left-symmetric conformal algebras. For example, the theory of left-symmetric conformal bialgebras was given in \cite{HL1}, the general cohomology theory was presented in \cite{ZH} and compatible left-symmetric conformal algebra structures on the Lie conformal algebra $W(a,b)$ were investigated in \cite{LHZZ, WH}.

Obviously, the theory of finite left-symmetric conformal algebras is far from being developed, for example, there is no complete classification of finite simple left-symmetric conformal algebras up to now. As \cite{DK} showed, a finite simple Lie conformal algebra is either of rank $1$  or isomorphic to a current Lie conformal algebra associated with a finite-dimensional simple Lie algebra. However, different from the classification of Lie conformal algebras, there are simple left-symmetric conformal algebras of rank $2$ which are not current (see Example \ref{rank 2}). Therefore, it seems hard to give a complete classification of finite simple left-symmetric conformal algebras. Note that similar to quadratic Lie conformal algebras \cite{X1}, the definition of quadratic left-symmetric conformal algebra was given in \cite{HL}. It was shown in \cite{HL} that a quadratic left-symmetric conformal algebra $R=\mathbb{C}[\partial]V$ is equivalent to a quadruple $(V, \ast_1, \circ, \ast_2)$, where $(V, \circ)$ is a left-symmetric algebra, and $\ast_1$, $\ast_2$, $\circ$ satisfy $9$ identities. Motivated by the study of simplicities of quadratic Lie conformal algebras in \cite{HW}, a natural idea is to investigate the simplicities of quadratic left-symmetric conformal algebras, from which we can construct and provide many finite simple left-symmetric conformal algebras.

On the other hand, the study of central extensions is also very important in the classification of finite left-symmetric conformal algebras. Since $\mathbb{C}[\partial]$ is a principal ideal domain, as the case of Lie conformal algebras in \cite{K1}, the characterization of finite left-symmetric conformal algebras can be attributed to the following problems:
\begin{itemize}
\item Classify finite left-symmetric conformal algebras which are free as $\mathbb{C}[\partial]$-modules.
\item Characterize central extensions of the obtained finite left-symmetric conformal algebras with the centre in the torsion.
\end{itemize}
Motivated by the study of central extensions of quadratic Lie conformal algebras in \cite{H}, it is natural and meaningful to investigate the central extensions of quadratic left-symmetric conformal algebras.

In the study of simplicity and central extensions of quadratic left-symmetric conformal algebras, there is a problem that we should deal with first, i.e. the algebra structure $(V, \ast_1, \circ, \ast_2)$ is too complicated to investigate. Motivated by the definition of pre-Novikov algebra given in \cite{HBG}, which is equivalent to a special class of quadratic left-symmetric conformal algebras (see \cite{HBG1}), we introduce the definition of pre-Gel'fand-Dorfman algebra and show that a quadratic left-symmetric conformal algebra $R=\mathbb{C}[\partial]V$ is equivalent to a pre-Gel'fand-Dorfman algebra $(V, \lhd, \rhd, \circ)$. Note that left-symmetric Poisson algebras and Novikov-Poisson algebras are pre-Gel'fand-Dorfman algebras. Based on this correspondence, we investigate the simplicities and central extensions of quadratic left-symmetric conformal algebras by a one-dimensional centre $\mathbb{C}\mathfrak{c}_\beta$, where $\partial \mathfrak{c}_\beta=\beta \mathfrak{c}_\beta$, $\beta\in \mathbb{C}$. Some necessary conditions and sufficient conditions for a quadratic left-symmetric conformal algebra to be simple are presented. In particular, we show that if a Novikov-Poisson algebra is simple, then the corresponding quadratic left-symmetric conformal algebra is simple. This can be used to construct many simple left-symmetric conformal algebras. In addition, we show that the central extensions of quadratic left-symmetric conformal algebras  by a one-dimensional centre $\mathbb{C}\mathfrak{c}_\beta$ in many cases are determined by four bilinear forms on the corresponding pre-Gel'fand-Dorfman algebras. This will facilitate us to calculate the central extensions of quadratic left symmetric conformal algebras  by a one-dimensional centre $\mathbb{C}\mathfrak{c}_\beta$. Several examples are also presented.

This paper is organized as follows. In Section 2, some basic definitions about left-symmetric algebras, left-symmetric conformal algebras, quadratic left-symmetric conformal algebras and pre-Novikov algebras are recalled. We introduce the definition of pre-Gel'fand-Dorfman algebra and show that a quadratic left-symmetric conformal algebra is equivalent to a pre-Gel'fand-Dorfman algebra. Some constructions of pre-Gel'fand-Dorfman algebras are given. In Section 3,   we study the central extensions of quadratic left-symmetric conformal algebras by a one-dimensional centre $\mathbb{C}\mathfrak{c}_\beta$.
Section 4 is devoted to investigating the simplicities of quadratic left-symmetric conformal algebras. Some necessary conditions and sufficient conditions for a quadratic left-symmetric conformal algebra to be simple are presented. Some examples of simple left-symmetric conformal algebras are also given.

Throughout this paper, we denote by $\mathbb{C}$, $\mathbb{Z}$ and
$\mathbb{Z}_+$ the sets of complex numbers,
integers and nonnegative integers, respectively. All vector spaces and tensor products are taken over the complex field $\mathbb{C}$. For any vector space $V$,
we use $V[\lambda]$ to denote the set of polynomials of $\lambda$ with coefficients in $V$.

\section{A new equivalent characterization of quadratic left-symmetric conformal algebras}

In this section, we will recall some basic definitions and facts about quadratic left-symmetric conformal algebras \cite{HL} and give a new equivalent characterization of quadratic left-symmetric conformal algebras by introducing the definition of pre-Gel'fand-Dorfman algebra. Some constructions of  pre-Gel'fand-Dorfman algebras are also given.

We first recall the definitions of left-symmetric algebra and Novikov algebra.
\begin{defi}
 A {\bf left-symmetric algebra} $A$ is a vector space over $\mathbb{C}$ with a bilinear product ``$\circ $": $A\times A\rightarrow A$, which
 satisfies the following condition:
\begin{align}
(a\circ b)\circ c-a\circ(b\circ c)=(b\circ a)\circ c-b\circ(a\circ c),\label{lsalgebra}\;\;\;\text{for all $a$, $b$, $c\in A$.}
\end{align}
\delete{We denote it by $(A, \circ)$.}

If the product `` $\circ $ " also satisfies (for all $a,b,c\in A$):
\begin{align}
(a\circ b)\circ c=(a\circ c)\circ b,\label{nalgebra}
\end{align}
then $(A, \circ)$ is called a {\bf Novikov algebra}.
\end{defi}

\delete{\begin{rmk}
 The Novikov algebra was essentially stated in \cite{gel1979hamiltonian}. It corresponds to a certain Hamiltonian operator and also appeared in \cite{balinskii1985poisson} from
the point of view of Poisson structures of hydrodynamic type.
\end{rmk}}

Next, let us recall the definition of left-symmetric conformal algebra.

\begin{defi} \cite{HL}
A {\bf left-symmetric conformal algebra} $R$ is a $\mathbb{C}\left[\partial\right]$-module with a $\lambda$-product $\cdot_\lambda\cdot$
which defines a $\mathbb{C}$-bilinear map from $R\times R \rightarrow R\left[\lambda\right]$, satisfying
\begin{eqnarray*}
&&\partial a_\lambda b=-\lambda a_\lambda b,~~~~~~~~a_\lambda\partial b=(\partial+\lambda)a_\lambda b,\;\;\;\text{(conformal sesquilinearity)}\\
&&(a_\lambda b)_{\lambda+\mu}c-a_\lambda(b_\mu c)=(b_\mu a)_{\lambda+\mu}c-b_\mu(a_\lambda c),\;\;\;\text{(left-symmetry)}
\end{eqnarray*}
for all $a,b,c\in R$. We denote it by $(R, \cdot_\lambda \cdot)$.

\end{defi}

A left-symmetric conformal algebra  is said to be {\bf finite}, if it is finitely generated as a $\mathbb{C}\left[\partial\right]$-module. Otherwise, we call
it {\bf infinite}. A $\mathbb{C}[\partial]$-submodule $I$ of a left-symmetric conformal algebra $R$ is called an {\bf ideal} if $I_\lambda R\subseteq I[\lambda]$ and $R_\lambda I\subseteq I[\lambda]$. A left-symmetric conformal algebra $R$ is called {\bf simple} if $R$ is non-trivial and $R$ has no proper ideals.

Let $(R, \cdot_\lambda \cdot)$ be a left-symmetric conformal algebra. Set $a_\lambda b=\sum_{n=0}^\infty \frac{\lambda^n}{n!} a_{(n)}b$ for any $a$, $b\in R$, where $a_{(n)}b\in R$. Let $\text{Coeff}(R)$ be the quotient of the vector space with basis $a_n$ $(a\in R,
n\in \mathbb{Z})$ by the subspace spanned over $\mathbb{C}$ by elements:
\begin{eqnarray*}
(\alpha a)_n-\alpha a_n,\;\;(a+b)_n-a_n-b_n,\;\;(\partial a)_n+na_{n-1},\;\;\text{where $a$, $b\in R$, $\alpha\in \mathbb{C}$, $n\in \ZZ$.}
\end{eqnarray*}
Define an operation on $\text{Coeff}(R)$ as follows:
\begin{eqnarray}
a_m \circ b_n=\sum_{j\in \ZZ_+}\left(
                                 \begin{array}{c}
                                   m\\
                                   j \\
                                 \end{array}
                               \right)(a_{(j)}b)_{m+n-j}.
\end{eqnarray}
Then $(\text{Coeff} (R), \circ)$ is a left-symmetric algebra (see \cite{HL}).

\begin{ex}\label{curl}
Let $(A,\circ) $ be a left-symmetric algebra. Then we can naturally define a left-symmetric conformal algebra $\text{Cur} L=\mathbb{C}[\partial]\otimes L$ with
the $\lambda$-product:
$$
a_\lambda b=a\circ b,\;\;\;\text{for all $a,b\in L$.}
$$
$\text{Cur} L$ is called the {\bf current left-symmetric conformal algebra} associated with $L$.
\end{ex}

\begin{ex}\cite{HL}\label{rank one}
Let $R=\mathbb{C}[\partial]L$ be a left-symmetric conformal algebra of rank one as a $\mathbb{C}[\partial]$-module with the $\lambda$-product:
$$L_\lambda L=(\partial+\lambda+c)L,\;\;\;\text{for some $c\in \mathbb{C}$.}
$$
We denote it by $ R_c$.
\end{ex}

Next, we will introduce a class of special left-symmetric conformal algebras.

\begin{defi}\cite{HL}
$(R, \cdot_\lambda \cdot)$ is called a {\bf quadratic left-symmetric conformal algebra} if there exists some vector space $V$ such that $R=\mathbb{C}[\partial]V $ and for any $u,v\in V$,
\begin{equation}
u_\lambda v=\partial w_1+w_2+\lambda w_3,\label{uv}
\end{equation}
where $w_1,w_2,w_3\in V$.
\end{defi}

\begin{pro}\label{ql}\cite[Theorem 3.7]{HL}
A quadratic left-symmetric conformal algebra $R=\mathbb{C}\left[\partial\right] V$ with the $\lambda$-product:
$$
a_\lambda b=\partial(a*_1 b)+a\circ b+\lambda(a*_2 b), ~~\text{for all}~~a, b\in V,
$$
is equivalent to the quadruple $(V, \ast_{1}, \circ, \ast_{2})$
where `` $\ast_{1} "$ and `` $\ast_{2} "$ are two operations on $V$, $(V, \circ)$ is a left-symmetric algebra, and they satisfy the following compatibility conditions:
for all $a,b,c\in V$,
\small{
\begin{eqnarray}
&&\label{oe1}a\ast_{1}(b \ast_{1} c)=b \ast_{1}(a \ast_{1} c),\\
&&\label{oe2}(a *_{1} b) *_{1} c-(a *_{2} b) *_{1} c+a *_{1}(b *_{1} c)+a *_{2}(b *_{1} c)=(b *_{1} a) *_{1} c+b *_{1}(a *_{2} c),\\
&& \label{oe3}(a\ast_1 b)\ast_1 c+a\ast_1 (b\ast_2 c)=(b\ast_1 a)\ast_1 c-(b\ast_2a)\ast_1 c+b\ast_1 (a\ast_1 c)+b\ast_2 (a\ast_1 c),\\
&&\label{oe4}(a *_{1} b) *_{2} c-(a *_{2} b) *_{2} c+a *_{2}(b *_{1} c)=(b *_1 a)*_2c,\\
&&\label{oe5}2(a *_{1} b) *_{2} c-(a *_{2} b) *_{2} c+a *_{2}(b *_{2} c)
=2(b *_{1} a) *_{2} c-(b *_{2} a) *_{2} c+b *_{2}(a *_{2} c),\\
&&\label{oe6}(a\ast_1 b)\ast_2 c=(b\ast_1 a)\ast_2 c-(b\ast_2 a)\ast_2 c+b\ast_2 (a\ast_1 c),\\
&&\label{oe7}(a \circ b) *_{1} c-a \circ(b *_{1} c)-a *_{1}(b \circ c)=(b \circ a) *_{1} c-b \circ(a *_{1} c)-b *_{1}(a \circ c),\\
&&\label{oe8}(a *_{1} b) \circ c-(a \circ b) *_{2} c-(a *_{2} b) \circ c+a \circ(b *_{1} c)+a *_{2}(b \circ c)\\
&&~~=(b *_{1} a) \circ c-(b \circ a) *_{2} c+b \circ(a *_{2} c),\nonumber\\
&& \label{oe9}(a\ast_1 b)\circ c-(a\circ b)\ast_2 c+a\circ (b\ast_2 c)\\
&&~~=(b\ast_1 a)\circ c-(b\circ a)\ast_2 c-(b\ast_2 a)\circ c+b\circ (a\ast_1 c)+b\ast_2 (a\circ c).\nonumber
\end{eqnarray}
}
\end{pro}

 In order to study quadratic left-symmetric conformal algebra better, we recall the definition of pre-Novikov algebra.
\begin{defi}\label{ND} \cite{HBG}
Let $A$ be a vector space with binary operations `` $ \lhd$ " and `` $\rhd$ ". If for all $a$, $b$ and $c\in A$, they satisfy the following equalities
\begin{align}
&a\rhd(b\rhd c)=(a\rhd b+a\lhd b)\rhd c+b\rhd(a\rhd c)-(b\rhd a+b\lhd a)\rhd c,\label{ND1}\\
&a\rhd(b\lhd c)=(a\rhd b)\lhd c+b\lhd(a\lhd c+a\rhd c)-(b\lhd a)\lhd c,\label{ND2}\\
&(a\lhd b+a\rhd b)\rhd c=(a\rhd c)\lhd b,\label{ND3}\\
&(a\lhd b)\lhd c=(a\lhd c)\lhd b,\label{ND4}
\end{align}
then $(A,\lhd,\rhd)$ is called a \bf{pre-Novikov algebra}.
\end{defi}

Recall that
\begin{defi}\cite{O}
A {\bf representation} of a Novikov algebra $(A,\ast)$ is a triple $(M, l_A,r_A)$, where $M$ is a vector space and   $l_A$, $r_A: A\rightarrow {\rm
End}_{\mathbb{C}}(M)$ are linear maps satisfying
\begin{eqnarray}
\mlabel{lef-mod1}&l_A(a\ast b-b\ast a)v=l_A(a)l_A(b)v-l_A(b)l_A(a)v,&\\
\mlabel{lef-mod2}&l_A(a)r_A(b)v-r_A(b)l_A(a)v=r_A(a\ast b)v-r_A(b)r_A(a)v,&\\
\mlabel{Nov-mod1}&l_A(a\ast b)v=r_A(b)l_A(a)v,&\\
\mlabel{Nov-mod2}&r_A(a)r_A(b)v=r_A(b)r_A(a)v,&
\end{eqnarray}
for all $a$, $b\in A$ and $v\in M$.
\end{defi}

For a pre-Novikov algebra $(A, \lhd, \rhd)$, define linear maps $L_\rhd$, $R_\lhd: A\rightarrow \text{End}_{\mathbb{C}}(A)$ by
\begin{eqnarray*}
L_\rhd (a)(b):=a\rhd b,\;\;\;R_{\lhd}(a)(b):=b\lhd a, \;\; \; \text{for all $a$, $b\in A$.}
\end{eqnarray*}

\begin{pro}\cite[Proposition 3.31]{HBG}\label{pre-Nov-Nov}
Let $(A, \lhd, \rhd)$ be a pre-Novikov algebra. The binary operation
\begin{eqnarray}
\label{ND5}
\ast: A\otimes A\rightarrow A, \quad
a \ast  b\coloneqq a\lhd b+a\rhd b~~~~\;\;\tforall  a, b\in A,
\end{eqnarray}
defines a Novikov algebra, which is called the {\bf associated Novikov algebra} of $(A, \lhd, \rhd)$.
Moreover,  $(A, L_{\rhd}, R_{\lhd})$ is
a representation of $(A, \ast)$. Conversely, let $A$ be a vector space
with binary operations $\rhd$ and $\lhd$. If $(A, \ast)$ defined by Eq.~\meqref{ND5} is a Novikov
algebra and $(A, L_{\rhd}, R_{\lhd})$ is a representation of $(A,
\ast)$, then $(A, \lhd, \rhd)$ is a pre-Novikov algebra.
\end{pro}
\begin{rmk}
By Proposition \ref{pre-Nov-Nov}, the operad of pre-Novikov algebras is the successor of the operad of Novikov algebras in the sense of \cite{BBGN}.
\end{rmk}

\begin{defi} \cite{Lo}
Let $A$ be a vector space. If there is a binary operation `` $\cdot $ " on $A$ satisfying
\begin{equation}
a\cdot(b\cdot c)=(a\cdot b+b\cdot a)\cdot c ,\nonumber
\end{equation}
for all $a, b, c\in A$ , then $(A,\cdot)$ is called a {\bf Zinbiel algebra}.
\end{defi}
\begin{rmk}
Note that for a Zinbiel algebra $(A, \cdot)$, for all $a, b, c\in A$,
\begin{eqnarray*}
a \cdot (b\cdot c)=b\cdot (a\cdot c).
\end{eqnarray*}
\end{rmk}
\begin{pro}\label{constr1}
Let $(A, \cdot)$ be a Zinbiel algebra, $D$ be a derivation of $(A, \cdot)$ and $\xi\in \mathbb{C}$. Define binary operations $\lhd$ and $\rhd$ on $A$ as follows:
\begin{eqnarray}
a\lhd b:=D(b)\cdot a+\xi b\cdot a,\;\; a\rhd b:= a\cdot D(b)+\xi a \cdot b, \;\;\text{for all $a$, $b\in A.$}
\end{eqnarray}
Then $(A, \lhd, \rhd)$ is a pre-Novikov algebra.
\end{pro}
\begin{proof}
 For all $a,b,c\in A$,
\begin{align*}
&a\rhd(b\rhd c)-(a\rhd b+a\lhd b)\rhd c-b\rhd(a\rhd c)+(b\rhd a+b\lhd a)\rhd c\\
=&a\rhd(b\cdot D(c)+\xi b \cdot c)-(a\cdot D(b)+\xi a \cdot b+D(b)\cdot a+\xi b\cdot a)\rhd c-b\rhd(a\cdot D(c)+\xi a \cdot c)\\
&+(b\cdot D(a)+\xi b \cdot a+D(a)\cdot b+\xi a\cdot b)\rhd c\\
=&a\cdot D( b\cdot D(c)+\xi b \cdot c)+\xi a\cdot(b\cdot D(c)+\xi b \cdot c) -(a\cdot D(b)+\xi a \cdot b+D(b)\cdot a+\xi b\cdot a)\cdot D(c)\\
&-\xi(a\cdot D(b)+\xi a \cdot b+D(b)\cdot a+\xi b\cdot a)\cdot c-b\cdot D(a\cdot D(c)+\xi a \cdot c)-\xi b\cdot(a\cdot D(c)+\xi a \cdot c)\\
&+(b\cdot D(a)+\xi b \cdot a+D(a)\cdot b+\xi a\cdot b)\cdot D(c)+\xi(b\cdot D(a)+\xi b \cdot a+D(a)\cdot b+\xi a\cdot b)\cdot c\\
=&a\cdot(D(b)\cdot D(c)+b\cdot D^2(c)+\xi D(b)\cdot c+\xi b\cdot D(c))+\xi a\cdot(b\cdot D(c))+\xi^2 a\cdot(b\cdot c)\\
&-(a\cdot D(b)+\xi a\cdot b+D(b)\cdot a+\xi b\cdot a)\cdot D(c)-\xi (a\cdot D(b)+\xi a \cdot b+D(b)\cdot a+\xi b\cdot a)\cdot c\\
&-b\cdot (D(a)\cdot D(c)+a\cdot D^2(c)+\xi D(a) \cdot c+\xi a\cdot D(c))-\xi b\cdot(a\cdot D(c)+\xi a \cdot c)\\
&+(b\cdot D(a)+\xi b \cdot a+D(a)\cdot b+\xi a\cdot b)\cdot D(c)+\xi(b\cdot D(a)+\xi b \cdot a+D(a)\cdot b+\xi a\cdot b)\cdot c\\
=& (a\cdot (D(b)\cdot D(c))-(a\cdot D(b)+D(b)\cdot a)\cdot D(c))+(a\cdot (b\cdot D^2(c))-b\cdot (a\cdot D^2(c))\\
&+((b\cdot D(a)+D(a)\cdot b)\cdot c-b\cdot (D(a)\cdot D(c)))+\xi(a\cdot (D(b)\cdot c)-(a\cdot D(b)+D(b)\cdot a)\cdot c)\\
&+\xi(a\cdot (b\cdot D(c))-(a\cdot b+b\cdot a)\cdot D(c))-\xi (b\cdot (D(a)\cdot c)-(b\cdot D(a)+D(a)\cdot b)\cdot c)\\
&+\xi(a\cdot (b\cdot D(c))-b\cdot (a\cdot D(c)))-\xi (b\cdot (a\cdot D(c))-(b\cdot a+a\cdot b)\cdot D(c))\\
&+\xi^2(a\cdot (b\cdot c)-(a\cdot b+b\cdot a)\cdot c)-\xi^2(b\cdot (a\cdot c)-(b\cdot a+a\cdot b)\cdot c)\\
=&0.
\end{align*}
Therefore (\ref{ND1}) holds. Other equalities can be check similarly.
\end{proof}

\begin{rmk}
Note that the construction given in Proposition \ref{constr1} when $\xi=0$ was given in \cite{HBG}.
\end{rmk}

Next, we introduce a class of new algebras named pre-Gel'fand-Dorfman algebras.
\begin{defi}
Let $(A, \lhd, \rhd)$ be a pre-Novikov algebra and $(A, \circ)$ be a left-symmetric algebra. If they satisfy the following compatibility conditions
\begin{eqnarray}
&&c\lhd(a\circ b-b\circ a)-a\circ(c\lhd b)-(b\circ c)\lhd a=-b\circ(c\lhd a)-(a\circ c)\lhd b,\label{lnd1}\\
&&(a\circ b-b\circ a)\rhd c+(a\lhd b+a\rhd b)\circ c
=a\rhd(b\circ c)-b\circ(a\rhd c)+(a\circ c)\lhd b,\label{lnd2}
\end{eqnarray}
for all $a,b,c\in V$, then this quadruple
$(V,\lhd,\rhd,\circ)$ is called a {\bf pre-Gel'fand-Dorfman algebra}.
\end{defi}

In order to understand this definition, we recall the definitions of Gel'fand-Dorfman algebra and its representation.

\begin{defi}\cite{X1}
Let $(A, \ast)$ be a Novikov algebra and $(A, [\cdot, \cdot])$ be a Lie algebra. If they also satisfy the following compatibility condition
\begin{eqnarray}
[a\ast b, c]-[a\ast c,b]+[a,b]\ast c-[a, c]\ast b-a\ast [b,c]=0,\;\;\text{ for all $a$, $b$, $c\in A,$}
\end{eqnarray}
then $(A, \ast, [\cdot, \cdot])$ is called a {\bf Gel'fand-Dorfman algebra}.
\end{defi}
\begin{rmk}
Note that such algebra is called Gel'fand-Dorfman bialgebra in \cite{X1}. To avoid a confusion with the definition of usual bialgebra which is composed by an algebra and a co-algebra, we call it Gel'fand-Dorfman algebra as \cite{KSO}.  \end{rmk}

\begin{defi}\cite{WenH}
Let $(A, \ast, [\cdot, \cdot])$ be a Gel'fand-Dorfman algebra and $V$ be a vector space, together with a bilinear map  $\rho_A $: $A \to \text{End}_{\mathbb{C}}(V)$ and two linear maps $l_A$,$\ r_A$: $A \to \text{End}_{\mathbb{C}}(V)$. Then $(V, l_A, r_A, \rho_A)$ is called a {\bf representation} of $(A, \ast, [\cdot, \cdot])$ if $(V,\rho_A)$ is a representation of the Lie algebra $(A, [\cdot,\cdot])$, $(V, l_A, r_A)$ is a representation of the Novikov algebra $(A, \ast)$, and
\begin{eqnarray*}
&&\rho_A(a)l_A(b)v+\rho_A(b\ast a) v+l_A([b,a])v-r_A(a)\rho_A(b)v-l_A(b)\rho_A(a)v=0,\\
&&\rho_A(a)r_A(b)v-\rho_A(b)r_A(a)v-r_A(b)\rho_A(a)v+r_A(a)\rho_A(b)v-r_A([a,b])v =0,
\end{eqnarray*}
for all $a$, $b\in A$, $v\in V$.
\end{defi}

We give the relationship between Gel'fand-Dorfman algebras and pre-Gel'fand-Dorfman algebras as follows.
\begin{pro}\label{GD-PGD}
Let $(A,\lhd,\rhd, \circ)$ be a pre-Gel'fand-Dorfman algebra. Define
\begin{eqnarray}\label{GD1}
a*b=a\lhd b+a\rhd b,~~\left[a ,b\right]=a\circ b-b\circ a,\;\;\;\text{for all $a$, $b\in A$.}
\end{eqnarray}
Then $(A,*,\left[\cdot,\cdot\right])$ is a Gel'fand-Dorfman algebra, which is called the {\bf associated Gel'fand-Dorfman algebra} of $(A, \lhd, \rhd, \circ)$.
Moreover,  $(A,  L_{\rhd}, R_{\lhd}, L_\circ)$ is
a representation of the Gel'fand-Dorfman algebra $(A, \ast, [\cdot, \cdot])$.

Conversely, let $A$ be a vector space
with binary operations $\rhd$, $\lhd$ and $\circ$. If $(A, \ast, [\cdot, \cdot])$ defined by Eq.~\meqref{GD1} is a Gel'fand-Dorfman algebra and $(A,  L_{\rhd}, R_{\lhd}, L_\circ)$ is a representation of $(A, \ast, [\cdot, \cdot])$, then $(A, \lhd, \rhd, \circ)$ is a pre-Gel'fand-Dorfman algebra.
\end{pro}
\begin{proof}
It is straightforward to check.
\end{proof}

\begin{rmk}
By Proposition \ref{GD-PGD}, the operad of pre-Gel'fand-Dorfman algebras is the successor of the operad of Gel'fand-Dorfman algebras in the sense of \cite{BBGN}.
\end{rmk}

Then Proposition \ref{ql} can be revised as follows using pre-Gel'fand-Dorfman algebras.

\begin{thm}\label{lspn}
A quadratic left-symmetric conformal algebra  $R=\mathbb{C}\left[\partial\right]V$ with the $\lambda$-products
\begin{equation}
a_\lambda b=\partial(b\lhd a)+a\circ b+\lambda(a\rhd b+b\lhd a),\label{lnd}\;\; \text{for all $a$, $b\in V$},
\end{equation}
is equivalent to a pre-Gel'fand-Dorfman algebra
$(V,\lhd,\rhd,\circ)$. We say that $R=\mathbb{C}[\partial]V$ is the {\bf quadratic left-symmetric conformal algebra corresponding to $(V,\lhd,\rhd,\circ)$}.
\end{thm}
\begin{proof}
By Proposition \ref{ql}, we set $a\ast_1 b=b\lhd a$, $a\ast_2 b=a\rhd b+b\lhd a$ for all $a$, $b\in A$. One can directly check that (\ref{oe1})-(\ref{oe6}) are equivalent to that $(V, \lhd, \rhd)$ is a pre-Novikov algebra,
and  (\ref{oe7})-(\ref{oe9}) are equivalent to (\ref{lnd1}) and (\ref{lnd2}). Then the proof is finished.
\end{proof}

\delete{\begin{rmk}
Obviously, the $\lambda$-product of a quadratic left-symmetric conformal algebra can be always written as the form given by Eq. (\ref{lnd}). Therefore, a quadratic left-symmetric conformal algebra is equivalent to a left-symmetric pre-Novikov algebra.
\end{rmk}}

\begin{rmk}
Let $R=\mathbb{C}[\partial]V$ be the quadratic left-symmetric conformal algebra  corresponding to a pre-Gel'fand-Dorfman algebra $(V, \lhd, \rhd, \circ)$. By the definition of coefficient algebra, $\text{Coeff}(R)$ can be seen as $V\otimes \mathbb{C}[t,t^{-1}]$ with products given by
\begin{eqnarray*}
(a\otimes t^m) \circ (b\otimes t^n)=m(a\rhd b)\otimes t^{m+n-1}-n(b\lhd a)\otimes t^{m+n-1}+(a\circ b)\otimes t^{m+n},\;\;a, b\in V, m, n\in \mathbb{Z}.
\end{eqnarray*}
\end{rmk}

There is a natural construction of pre-Gel'fand-Dorfman algebras via pre-Novikov algebras.
\begin{pro}\label{Cons-pre-GD}
Let $(A, \lhd, \rhd)$ be a pre-Novikov algebra. Define the operation `` $\circ$ " on $A$ by
\begin{equation}
a\circ b:=k(a\rhd b-b\lhd a),\;\;\;\;\text{for all $a$, $b\in A$ and some fixed $k\in\mathbb{C}$}.
\end{equation}
Then $(A, \lhd, \rhd, \circ)$ is a pre-Gel'fand-Dorfman algebra.
\end{pro}
\begin{proof}
We only need to check that $(A, \circ)$ is a left-symmetric algebra, and \eqref{lnd1}, \eqref{lnd2} hold.
Let $a$, $b$, $c\in A$. Firstly, we check that $(A, \circ)$ is a left-symmetric algebra.
\begin{align*}
&(a\circ b)\circ c-a\circ(b\circ c)-(b\circ a)\circ c+b\circ(a\circ c)\\
=&k^2\big((a\rhd b-b\lhd a)\rhd c-c\lhd(a\rhd b-b\lhd a)-a\rhd(b\rhd c-c\lhd b)+(b\rhd c-c\lhd b)\lhd a\\
&-(b\rhd a-a\lhd b)\rhd c+c\lhd(b\rhd a-a\lhd b)+b\rhd(a\rhd c-c\lhd a)-(a\rhd c-c\lhd a)\lhd b\big)\\
=&k^2\Big (\big ((a\rhd b+a\lhd b)\rhd c-a\rhd(b\rhd c)+b\rhd(a\rhd c)-(b\rhd a+b\lhd a)\rhd c  \big)\\
&+(c\lhd(b\lhd a+b\rhd a)-b\rhd(c\lhd a)+(b\rhd c-c\lhd b)\lhd a)\\
&-(c\lhd(a\lhd b+a\rhd b)-a\rhd(c\lhd b)+(a\rhd c-c\lhd a)\lhd b) \Big) \\
=&0.
\end{align*}
Secondly, we check \eqref{lnd1}.
\begin{align*}
&c\lhd (a\circ b-b\circ a)-a\circ (c\lhd b)-(b\circ c)\lhd a +b\circ (c\lhd a)+(a\circ c)\lhd b\\
=&k(c\lhd(a\rhd b-b\lhd a-b\rhd a+a\lhd b)-a\rhd(c\lhd b)+(c\lhd b)\lhd a-(b\rhd c-c\lhd b)\lhd a\\
&+b\rhd(c\lhd a)-(c\lhd a)\lhd b+(a\rhd c-c\lhd a)\lhd b)\\
=&k\big((c\lhd(a\lhd b+a\rhd b)-a\rhd(c\lhd b)+(a\rhd c-c\lhd a)\lhd b)\\
&-(c\lhd(b\lhd a+b\rhd a)-b\rhd(c\lhd a)+(b\rhd c-c\lhd b)\lhd a)\big )\\
=&0.
\end{align*}
Similarly, \eqref{lnd2} can be checked.
\end{proof}

Next, we will use Proposition \ref{Cons-pre-GD} to present a construction of pre-Gel'fand-Dorfman algebras via Zinbiel algebras.

\begin{cor}\label{corollarylsnd}
Let $(A,\cdot)$ be a Zinbiel algebra, $D$ be a derivation on $(A,\cdot)$ and $\xi$, $k\in \mathbb{C}$. Define
\begin{eqnarray*}
&&a\lhd b:=D(b)\cdot a+\xi  b\cdot a ,~~a\rhd b:=a\cdot D(b)+\xi a \cdot b,\\
&&a\circ b:=k(a\cdot D(b)- D(a)\cdot b+\xi (a\cdot b-b\cdot a)), 
\end{eqnarray*}
 for all $a, b\in A$. Then $(A,\lhd,\rhd, \circ)$ is a pre-Gel'fand-Dorfman algebra.
\end{cor}

We recall the definitions of left-symmetric Poisson algebra and Novikov-Poisson algebra.

\begin{defi} \label{alg} \cite{HL, X2}
Given a vector space $A$ with two binary operations  `` $\circ$ " and `` $\cdot $ ".
The triple $(A,\cdot,\circ)$ is called a \textbf{left-symmetric Poisson algebra} if  $(A, \circ)$  is a left-symmetric algebra,
 $(A, \cdot)$ is a commutative associative algebra, and they satisfy the compatibility conditions
\begin{align}
&(a\cdot b)\circ c=a\cdot (b\circ c),\label{lp1}\\
&(a\circ b)\cdot c-a\circ(b\cdot c)=(b\circ a)\cdot c-b\circ(a\cdot c), ~~~\text{for all $a$, $b$, $c\in A$}.\label{lp2}
\end{align}
In particular, if $(A, \circ)$ is a Novikov algebra, then
the triple $(A,\cdot,\circ)$ is called a \textbf{Novikov-Poisson algebra}.
\end{defi}

\begin{pro}
Let $(A, \cdot, \circ)$ be a left-symmetric Poisson algebra. Define
\begin{eqnarray*}
a\lhd b:=a\cdot b,\;\; a\rhd b:=0,\;\; \text{for all $a$, $b\in A$.}
\end{eqnarray*}
Then $(A, \lhd, \rhd, \circ)$ is a pre-Gel'fand-Dorfman algebra.
\end{pro}
\begin{proof}
It is straightforward or one can refer to \cite[Corollary 3.9]{HL}.
\end{proof}

Finally, we present an example of Novikov-Poisson algebra.
\begin{ex}\label{exp1}
Let $(V, \cdot)$ be a commutative associative algebra, and $D$ be a derivation of $(V, \cdot)$. Define the operation `` $\circ$ " on $V$ by
\begin{eqnarray*}
x\circ y=x\cdot D(y),\;\; \text{for all $x$, $y\in V$.}
\end{eqnarray*}
 Then $(V,\cdot,\circ)$ forms a Novikov-Poisson algebra. This construction was presented in \cite[Lemma 2.1] {X2}.

 For example, let $(V=\mathbb{C}[t,t^{-1}], \cdot)$ be the Laurent polynomial algebra and $D$ be the derivation defined by $D(t^i)=it^i$ for all $i\in \ZZ$. Define the operation `` $\circ$ " on $V$ by
 \begin{eqnarray*}
t^i\circ t^j=t^i\cdot D(t^j)=jt^{i+j},\;\;  \text{for all $i$, $j\in \ZZ$.}
\end{eqnarray*}
Then $(V, \cdot, \circ)$ is a Novikov-Poisson algebra.
\end{ex}


\delete{Let $a$, $b$, $c\in A$. We check Eq. (\ref{lnd1}).
\begin{eqnarray*}
&&c\lhd (a\circ b-b\circ a)-a\circ (c\lhd b)-(b\circ c)\lhd a +b\circ (c\lhd a)+(a\circ c)\lhd b\\
&=&kD(a\cdot D(b)-D(a)\cdot b-b\cdot D(a)+kD(b)\cdot a)\cdot c-ka\cdot D(D(b)\cdot c)+kD(a)\cdot (D(b)\cdot c)\\
&&-kD(a)\cdot (b\cdot D(c)-D(b)\cdot c)+kb\cdot D(D(a)\cdot c)\\
&&-kD(b)\cdot (D(a)\cdot c)+kD(b)\cdot (a\cdot D(c)-D(a)\cdot c)\\
&=&k(a\cdot D^2(b)-D^2(a)\cdot b-b\cdot D^2(a)+D^2(b)\cdot a)\cdot c-ka\cdot (D^2(b)\cdot c+D(b)\cdot D(c))\\
&&+kD(a)\cdot (D(b)\cdot c)-kD(a)\cdot (b\cdot D(c))+kD(a)\cdot (D(b)\cdot c)+kb\cdot (D^2(a)\cdot c+D(a)\cdot D(c))\\
&&-kD(b)\cdot (D(a)\cdot c)+kD(b)\cdot (a\cdot D(c))-kD(b)\cdot (D(a)\cdot c)\\
&=&k\Big((a\cdot D^2(b)+D^2(b)\cdot a)\cdot c-a\cdot (D^2(b)\cdot c)\Big)+k\Big(-(D^2(a)\cdot b+b\cdot D^2(a))\cdot c+b\cdot(D^2(a)\cdot c)\Big)\\
&&-k(a\cdot (D(b)\cdot D(c))-D(b)\cdot (a\cdot D(c)))+2k(D(a)\cdot (D(b)\cdot c)-D(b)\cdot (D(a)\cdot c))\\
&&-k(D(a)\cdot (b\cdot D(c))-b\cdot (D(a)\cdot D(c)))\\
&=& 0.
\end{eqnarray*}
Similarly, Eq. \eqref{lnd2} can be checked.
\end{proof}

Finally,  we give an example of pre-Gel'fand-Dorfman algebra.

\begin{ex}\label{defzl}
Let $A=\mathbb{C}\left[x\right]$ as a vector space. We endow $A$ with a Zinbiel algebra structure as follows
\begin{equation}
\begin{aligned}
&x^i\cdot x^j=\tiny{\begin{pmatrix}
i+j-1\\
i\\
\end{pmatrix}}
 x^{i+j},
\end{aligned}
\end{equation}
for all $i,j\in\mathbb{Z}_+$.  We define $D(x^i)=ix^i$ then $D$ is a derivation of this Zinbiel algebra.
Consequently,  $(A,\lhd,\rhd,\circ)$ is a pre-Gel'fand-Dorfman algebra when $(A,\lhd,\rhd,\circ)$
satisfies \eqref{zinbiel2}.
\end{ex}}

\section{One-dimensional central extension of quadratic left-symmetric conformal algebras}

In this section, we will investigate the central extensions of quadratic left-symmetric conformal algebras by a one-dimensional centre $\mathbb{C}\mathfrak{c}_\beta $ with $\partial\mathfrak{c}_\beta=\beta\mathfrak{c}_\beta$. Denote $\mathfrak{c}_0$ by $\mathfrak{c}$.

\begin{defi} Let $R$ be a left-symmetric conformal algebra and $C$ be an abelian left-symmetric conformal algebra, i.e. $c_\lambda d=0$ for all $c$, $d\in C$. If there is a short exact sequence of left-symmetric conformal algebras
\begin{equation}
C\rightarrowtail \widehat{R}\twoheadrightarrow R,
\end{equation}
and $C_\lambda \widehat{R}=\widehat{R}_\lambda C =0$, then $\widehat{R}$ is called a {\bf central extension} of $R$ by $C$.

Let $\widehat{R}_1$ and $\widehat{R}_2$ be two central extensions of $R$ by $C$. We say that  $\widehat{R}_1$ is {\bf equivalent} to $\widehat{R}_2$ if there is a left-symmetric conformal algebra homomorphism $\varphi: \widehat{R}_1\rightarrow \widehat{R}_2$ such that the following diagram is commutative:
\begin{eqnarray}\label{diag}
\xymatrix{
  C ~\ar@{=}[d]  \ar@{>->}[r]   & \widehat{R}_1 \ar[d]^{\varphi} \ar@{->>}[r]    & R\ar@{=}[d]\\
  C ~\ar@{>->}[r]        & \widehat{R}_2 \ar@{->>}[r]               &R
}
\end{eqnarray}

\end{defi}

In the following, we investigate the case when $R$ is a quadratic left-symmetric conformal algebra and $C=\mathbb{C}\mathfrak{c}_\beta$.

Since $R$ is free as a $\mathbb{C}\left[\partial\right]$-module,
 we have $\widehat{R}=R\oplus \mathbb{C}\mathfrak{c}_\beta $ as a $\mathbb{C}[\partial]$-module.
Note that ${\mathfrak{c}_\beta}_\lambda \widehat{R}=\widehat{R}_\lambda \mathfrak{c}_\beta =0$. Set the $\lambda$-product on $\widehat{R}$ as follows:
\begin{eqnarray*}
\widehat{a_\lambda b}=a_\lambda b+\alpha_\lambda(a,b)\mathfrak{c}_\beta, \;\;\text{for all  $a$, $b\in R,$}
\end{eqnarray*}
where $\cdot_\lambda \cdot$ is the $\lambda$-product on $R$ and $\alpha_\lambda(\cdot,\cdot) $ is a $\mathbb{C} $-bilinear map from $R\times R $ to $\mathbb{C}[\lambda]$. By the definition of left-symmetric conformal algebra,
$\widehat{R}$
is a left-symmetric conformal algebra if and only if  $\alpha_\lambda(\cdot,\cdot) $ satisfies  the following conditions:
\begin{align}
&\alpha_\lambda(\partial a,b)=-\lambda\alpha_\lambda(a,b),\qquad \alpha_\lambda(a,\partial b)=(\lambda+\beta)\alpha_\lambda(a,b),\label{2}\\
&\alpha_{\lambda+\mu}(a_\lambda b,c)-\alpha_\lambda(a,b_\mu c)=\alpha_{\lambda+\mu}(b_\mu a,c)-\alpha_\mu(b,a_\lambda c),\label{4}
\end{align}
for all $a$, $b$, $c\in R$.
\delete{ If it is a central extension of the quadratic Novikov conformal algebra, then by (NovC3) $\alpha_\lambda(\cdot,\cdot)$ also hold:
\begin{equation}
\alpha_{\lambda+\mu}(a_\lambda b,c)=\alpha_{-\mu-\beta}(a_\mu c,b),\label{5}
\end{equation}
for all $a,b,c\in R$.}
Therefore, such $\widehat{R}$ are completely determined by those $\alpha_\lambda(\cdot,\cdot)$ satisfying (\ref{2}) and (\ref{4}). Then we say that $\alpha_\lambda(\cdot,\cdot)$ is {\bf equivalent } to $\alpha_\lambda^{'}(\cdot,\cdot)$ if the corresponding central extensions are equivalent.

Let $\widehat{R}_1$ and $\widehat{R}_2$ be two central extensions of $R$ by $\mathbb{C}\mathfrak{c}_\beta$ determined by $\alpha_\lambda(\cdot,\cdot)$ and $\alpha_\lambda^{'}(\cdot,\cdot)$ respectively. Let $\mathbb{C}_\beta=\mathbb{C}$  be a $\mathbb{C}[\partial]$-module where $\partial k=\beta k$ for all $k\in \mathbb{C}$. By the definition of equivalence of central extensions,
$\varphi$ in (\ref{diag}) must be as follows:
\begin{eqnarray*}
\varphi(a+c)=a+c+\phi(a)\mathfrak{c}_\beta,\;\; a\in R, c\in \mathbb{C}\mathfrak{c}_\beta,
\end{eqnarray*}
where $\phi: R\rightarrow \mathbb{C}_\beta$ is a $\mathbb{C}[\partial]$-module homomorphism. Moreover, $\varphi:\widehat{R}_1\rightarrow \widehat{R}_2$ is a homomorphism of left-symmetric conformal algebras if and only if $\phi$ satisfies $\alpha_\lambda(a,b)=\alpha'_\lambda(a,b)+\phi(a_\lambda b)$ for all $a$, $b\in R$. Therefore, $\alpha_\lambda(\cdot,\cdot)$ is equivalent to $\alpha_\lambda^{'}(\cdot,\cdot)$  if and only if $\phi$ satisfies $\alpha_\lambda(a,b)=\alpha'_\lambda(a,b)+\phi(a_\lambda b)$ for all $a$, $b\in R$.

 By the general cohomology theory of left-symmetric conformal algebra developed in \cite{ZH}, it is easy to see that $\alpha_\lambda(\cdot,\cdot) $ is a 2-cocycle in $H^2(R,\mathbb{C}_\beta )$
 and $\phi :$ $R\rightarrow \mathbb{C}_\beta$ is a 1-coboundary. Therefore, by the discussion above, central extensions of a left-symmetric conformal algebra $R$ which is free as a $\mathbb{C}[\partial]$-module by a one-dimensional centre $\mathbb{C}\mathfrak{c}_\beta$ up to equivalence are characterized by the second cohomology group $H^2(R,\mathbb{C}_\beta )$.

 In the sequel, for a pre-Gel'fand-Dorfman algebra $(V, \lhd, \rhd, \circ)$, we set
 \begin{eqnarray*}
 a\ast b=a\rhd b+a\lhd b,\;\; a\star b=a\rhd b+b\lhd a, \text{ for all $a$, $b\in V$.}
 \end{eqnarray*}
 Note that $(V, \ast)$ is a Novikov algebra.

\begin{thm}\label{lsndt1}
Let $R=\mathbb{C}\left[\partial\right]V$ be the quadratic left-symmetric conformal algebra corresponding to a pre-Gel'fand-Dorfman algebra $(V,\lhd,\rhd,\circ )$.
Let $\widehat{R}=R\oplus \mathbb{C}\mathfrak{c}_{\beta}$ be a central extension of $R$ by $\mathbb{C} \mathfrak{c}_{\beta}$ with the following $\lambda$-products
\begin{equation}\label{lsab}
\begin{aligned}
\widehat{a_\lambda b}=\partial(b\lhd a)+a\circ b+\lambda(a\star b)+\alpha_\lambda(a,b)\mathfrak{c}_\beta,\;\;\text{for all $a$, $b\in V$.}
\end{aligned}
\end{equation}
Suppose that $\alpha_\lambda(a,b)=\sum\limits_{i=0}^{n}\lambda^i\alpha_i(a,b)\in\mathbb{C}\left[\lambda\right]$ for all $a,b\in V$, with $\alpha_i(\cdot, \cdot): V\times V\rightarrow \mathbb{C}$ and $\alpha_n(a,b)\neq0$
for some $a,b\in V$. Then we have\\
(1) Suppose that $\alpha_\lambda(a,b)=\sum\limits_{i=0}^{3}\lambda^i\alpha_i(a,b)$ for all $a,b\in V$. Then for all $a$, $b$, $c\in V$,
\begin{eqnarray}
&&\alpha_3(a\ast b,c)=\alpha_3(b\ast a,c)=\alpha_3(a, c\lhd b)=\alpha_3(b, a\rhd c),\label{qqlsa}\\
&&\alpha_2(a\ast b,c)-\alpha_2(a,c\lhd b)-\beta\alpha_3(a,c\lhd b)-\alpha_3(a,b\circ c)=\alpha_3(b\circ a,c)-\alpha_3(a\circ b,c),\label{qlsa4}\\
&&2\alpha_2(a\ast b,c)-\alpha_2(b\ast a,c)-\alpha_2(a,b\star c)=3\alpha_3(b\circ a,c)-3\alpha_3(a\circ b,c),\label{qlsa5}\\
&&\alpha_1(a\ast b,c)-\alpha_1(a,c\lhd b)- \beta\alpha_2(a,c\lhd b)-\alpha_2(a,b\circ c)=\alpha_2(b\circ a,c)-\alpha_2(a\circ b,c),\label{qlsa6}\\
&&\alpha_1(a\ast b,c)-\alpha_1(b\ast a,c)-\alpha_1(a,b\star c)+\alpha_1(b,a\star c)=2\alpha_2(b\circ a,c)-2\alpha_2(a\circ b,c),\label{qlsa7}\\
&&\alpha_0(a\ast b,c)-\alpha_0(a,c\lhd b)+\alpha_0(b,a\star c)-\beta\alpha_1(a,c\lhd b)-\alpha_1(a,b\circ c)\label{qlsa8}\\
&&=\alpha_1(b\circ a,c)-\alpha_1(a\circ b,c),\nonumber\\
&&\alpha_0(a\circ b,c)-\beta\alpha_0(a,c\lhd b)-\alpha_0(a,b\circ c)=\alpha_0(b\circ a,c)-\beta\alpha_0(b,c\lhd a)-\alpha_0(b,a\circ c).\label{qlsa9}
\end{eqnarray}

\noindent(2) If $n>3$,
\begin{equation}
\begin{aligned}
\alpha_n(a\ast b,c)= \alpha_n(a, b\star c)=\alpha_n(a,b\lhd c)=\alpha_n(a,b\rhd c)=0, \;\;\;\text{for all $a,b,c\in V$.}
\end{aligned}
\end{equation}

(3) $\alpha_\lambda(\cdot,\cdot)$ and $\alpha'_\lambda(\cdot,\cdot)$ are equivalent if and only if there is a linear map $\phi$ : $V\rightarrow\mathbb{C}$ such that
\begin{equation}
\begin{aligned}
\alpha_\lambda(a,b)=\alpha'_\lambda(a,b)+\beta\phi(b\lhd a)+\lambda\phi(a\star b)+\phi(a\circ b), \;\;\; \text{for all $a,b\in V$.}
\end{aligned}
\end{equation}

\end{thm}
\begin{proof}
Note that by conformal sesquilinearity, the $\lambda$-products on $\widehat{R}$ are determined by those
$\widehat{a_\lambda b}$ for all $a$, $b\in V$. Let $a$, $b$, $c\in V$.
By \eqref{2}, \eqref{4} becomes

\begin{equation}
\begin{aligned}
&-\mu\alpha_{\lambda+\mu}(b\lhd a,c)+\lambda\alpha_{\lambda+\mu}(a\rhd b,c)+\alpha_{\lambda+\mu}(a\circ b,c)\\
&-(\lambda+\beta)\alpha_{\lambda}(a,c\lhd b)-\mu\alpha_{\lambda}(a,b\star c)-\alpha_{\lambda}(a,b\circ c)\\
&=-\lambda\alpha_{\lambda+\mu}(a\lhd b,c)+\mu\alpha_{\lambda+\mu}(b\rhd a,c)+\alpha_{\lambda+\mu}(b\circ a,c)\\
&-(\mu+\beta)\alpha_{\mu}(b,c\lhd a)-\lambda\alpha_{\mu}(b,a\star c)-\alpha_{\mu}(b,a\circ c).\label{qls1}
\end{aligned}
\end{equation}
Note that  $\alpha_\lambda(a,b)=\sum_{i=0}^{n}\lambda^i\alpha_i(a,b)$. If $n>3$, comparing the coefficients of $\lambda^2\mu^{n-1}$ and $\lambda^{n-1}\mu^2$,
we get
$$
\begin{aligned}
&n\alpha_n(a\ast b,c)-\frac{n(n-1)}{2}\alpha_n(b\ast a,c)=0,\\
&\frac{n(n-1)}{2}\alpha_n(a\ast b,c)-n\alpha_n(b\ast a,c)=0.
\end{aligned}
$$
Since $n>3$, we obtain $\alpha_n(a\ast b,c)=0 $ for all $a,b,c\in V$. Then by comparing coefficients of $\lambda\mu^{n}$ and $\lambda^{n+1}$, we obtain that for all $a,b,c\in V$,
$$\alpha_n(b,a\star c)=\alpha_n(a,c\lhd b)=0.$$
Therefore, $\alpha_n(a,b\rhd c)=0$ for all $a,b,c\in V$.

If $n\leq3$, taking $\alpha_\lambda(a,b)=\sum_{i=0}^{3}\lambda^i\alpha_i(a,b)$ into \eqref{qls1} and comparing the coefficients of
$\lambda^4$, $\lambda\mu^3$, $\lambda^2\mu^2$, $\lambda^3$, $\lambda^2\mu$, $\lambda^2$, $\lambda\mu$, $\lambda$ and $\lambda^0\mu^0$, we get \eqref{qlsa4}-\eqref{qlsa9} and
\begin{eqnarray}
&&\alpha_3(a\rhd b,c)-\alpha_3(a,c\lhd b)=-\alpha_3(a\lhd b,c),\label{qlsa1}\\
&&-3\alpha_3(b\lhd a,c)+\alpha_3(a\rhd b,c)=-\alpha_3(a\lhd b,c)+3\alpha_3(b\rhd a,c)-\alpha_3(b,a\star c),\\
&&-3\alpha_3(b\lhd a,c)+3\alpha_3(a\rhd b,c)=-3\alpha_3(a\lhd b,c)+3\alpha_3(b\rhd a,c).\label{qlsa3}
\end{eqnarray}
It is straightforward to check that (\ref{qlsa1})-(\ref{qlsa3}) are equivalent to (\ref{qqlsa}). Therefore, in this case, $\alpha_\lambda(\cdot,\cdot)$ is a 2-cocycle if and only if \eqref{qqlsa}-\eqref{qlsa9} hold.

Note that a $\mathbb{C}\left[\partial\right]$-module homomorphism $\varphi$: $R\rightarrow \mathbb{C}_\beta$ can be determined by the restricted linear map
$\varphi|_V$: $V\rightarrow \mathbb{C}_\beta$. Therefore, $\alpha(\cdot,\cdot)$ and $\alpha'(\cdot,\cdot)$ are equivalent if and only if there is a $\mathbb{C}$-linear
map $\phi$: $V\rightarrow \mathbb{C}_\beta$ such that
\begin{eqnarray*}
\alpha_\lambda(a,b)=\alpha'_\lambda(a,b)+\phi(a_\lambda b)=\alpha'_\lambda(a,b)+\beta\phi(b\lhd a)+\lambda\phi(a\star b)+\phi(a\circ b), \;\;\text{for all $a,b\in V$.}
\end{eqnarray*}

The proof is completed.
\end{proof}

\begin{rmk}
It should be pointed out that for a quadratic left-symmetric conformal algebra $R=\mathbb{C}[\partial]V$ when $V$ is infinite-dimensional, there may not exist an upper bound $n$ such that   $\alpha_\lambda(a,b)=\sum\limits_{i=0}^{n}\lambda^i\alpha_i(a,b)$ for all $a,b\in V$.
\end{rmk}
\begin{cor}\label{lsndc1}
Let $(V,\lhd,\rhd,\circ )$ be a finite-dimensional pre-Gel'fand-Dorfman algebra with
$V = V\ast V$ or $V=V\star V$ or $V=V\lhd V$ or $V=V\rhd V$. Let $R = \mathbb{C}\left[\partial\right]V $ be the corresponding quadratic left-symmetric conformal algebra. Suppose
that $\widehat{R} = R\oplus \mathbb{C}\mathfrak{c}_\beta$ be a central extension of $(R, \cdot_\lambda\cdot)$ with the $\lambda$-products defined by \eqref{lsab}.
Then for all $a, b \in V$, $\alpha_\lambda(a, b) = \sum\limits^{3}_{i=0}\lambda^i\alpha_i(a, b)$, in which all
$\alpha_i(\cdot, \cdot):$  $V \times V \rightarrow \mathbb{C} $ are bilinear forms satisfying \eqref{qqlsa}-\eqref{qlsa9}.
Furthermore, $\alpha_\lambda(\cdot, \cdot)$ and $\alpha'_\lambda(\cdot, \cdot)$ are equivalent if and only if $\alpha_i(\cdot, \cdot) = \alpha'_i(\cdot, \cdot)$ for $i=2, 3$ and there is a linear map $\phi:$  $ V \rightarrow\mathbb{C}$
such that
\begin{align}
&\alpha_1(a,b)=\alpha'_1(a,b)+\phi(a\rhd b+b\lhd a),\\
&\alpha_0(a,b)=\alpha'_0(a,b)+\beta\phi(b\lhd a)+\phi(a\circ b),\;\;\text{ for all $a, b \in V$.}
\end{align}
\end{cor}
\begin{proof}
Since $R = \mathbb{C}\left[\partial\right]V$ is  finite, there exists some $n\in \ZZ_+$ such that
$\alpha_\lambda(a,b) =\sum\limits_{i=0}^{n} \lambda^i\alpha_i(a,b)$ for all $a, b \in V$. Then this corollary follows directly from Theorem \ref{lsndt1}.
\end{proof}

\begin{cor}\label{corleft}
Let $(V,\cdot, \circ )$ be a finite-dimensional left-symmetric Poisson algebra with $V=V\cdot V$ and $R=\mathbb{C}\left[\partial\right]V$ be the
corresponding quadratic left-symmetric conformal algebra. Suppose $\widehat{R}=R\oplus \mathbb{C}\mathfrak{c}_\beta$ be a central extension of $(R,\cdot_\lambda\cdot)$
with the $\lambda$-products defined by
\begin{equation}
\widehat{a_\lambda b}=\partial(a\cdot b)+a\circ b+\lambda(a\cdot b)+\alpha_\lambda(a,b)\mathfrak{c}_\beta,
\end{equation}
where $a,b\in V$ and $\alpha_\lambda(a,b)\in\mathbb{C}\left[\lambda\right] $. Then for all $a,b\in V$, $\alpha_\lambda(a,b)=\sum\limits_{i=0}^{2}\lambda^i\alpha_i(a,b)$ ,
where all $\alpha_i(\cdot, \cdot) :$  $ V\times V\rightarrow \mathbb{C}$ are bilinear forms satisfying
\begin{eqnarray}
&&\alpha_2(a\cdot b,c)=\alpha_2(a,c\cdot b),\label{corl2}\\
&&\alpha_2(a\circ b,c)-\alpha_1(a,c\cdot b)- \beta\alpha_2(a,c\cdot b)-\alpha_2(a,b\circ c)=-\alpha_1(a\cdot b,c)+\alpha_2(b\circ a,c),\label{corl3}\\
&&2\alpha_2(a\circ b,c)-\alpha_1(a,c\cdot b)=2\alpha_2(b\circ a,c)
-\alpha_1(b,c\cdot a),\label{corl4}\\
&&\alpha_1(a\circ b,c)-\alpha_0(a,c\cdot b)-\beta\alpha_1(a,c\cdot b)-\alpha_1(a,b\circ c)\label{corl5}\\
&&=-\alpha_0(a\cdot b,c)
+\alpha_1(b\circ a,c)-\alpha_0(b,c\cdot a),\nonumber\\
&&\alpha_0(a\circ b,c)-\beta\alpha_0(a,c\cdot b)-\alpha_0(a,b\circ c)=\alpha_0(b\circ a,c)-\beta\alpha_0(b,c\cdot a)-\alpha_0(b,a\circ c),\label{corl6}
\end{eqnarray}
for all $a,b,c\in V$. \delete{Particularly, when $\beta=0$, \eqref{corl3}-\eqref{corl6} are identical with
\begin{align}
&\alpha_2(a\circ b,c)-\alpha_1(a,c\lhd b)-\alpha_2(a,b\circ c)=-\alpha_1(a\lhd b,c)
+\alpha_2(b\circ a,c),\label{corl03}\\
&-\alpha_1(b\lhd a,c)+2\alpha_2(a\circ b,c)-\alpha_1(a,c\lhd b)=-\alpha_1(a\lhd b,c)+2\alpha_2(b\circ a,c)
-\alpha_1(b,c\lhd a),\label{corl04}\\
&\alpha_1(a\circ b,c)-\alpha_0(a,c\lhd b)-\alpha_1(a,b\circ c)=-\alpha_0(a\lhd b,c)
+\alpha_1(b\circ a,c)-\alpha_0(b,c\lhd a),\label{corl05}\\
&\alpha_0(a\circ b,c)-\alpha_0(a,b\circ c)=\alpha_0(b\circ a,c)-\alpha_0(b,a\circ c),\label{corl06}
\end{align}
for any $a,b,c\in V$. } Moreover, $\alpha_\lambda(\cdot,\cdot)$ is equivalent to $\alpha'_\lambda(\cdot,\cdot)$ if and only if $\alpha_2(\cdot,\cdot)=\alpha'_2(\cdot,\cdot)$ and
there is a linear map $\phi:$ $V\rightarrow\mathbb{C}$ such that
\begin{align}
&\alpha_1(a,b)=\alpha'_1(a,b)+\phi(a\cdot b),\\
&\alpha_0(a,b)=\alpha'_0(a,b)+\beta\phi( a \cdot b)+\phi(a\circ b),\;\;\text{ for all $a, b \in V$.}
\end{align}

\end{cor}

\begin{proof}
It is straightforward from Corollary \ref{lsndc1}.
\end{proof}

\begin{cor}\label{corollaryb}
Let $(V,\lhd,\rhd,\circ)$ be a pre-Gel'fand-Dorfman algebra  with $V=V\ast V$ or $V=V\star V$ or $V=V\lhd V$ or $V=V\rhd V$ where $a\circ b=k(a\rhd b-b\lhd a)$ for all $a$, $b\in V$ and some fixed $k\in \mathbb{C}$. Let $R=\mathbb{C}\left[\partial\right]V$ be the
corresponding quadratic left-symmetric conformal algebra and $\widehat{R}=R\oplus \mathbb{C}\mathfrak{c}_\beta$ be a central extension of $(R,\cdot_\lambda\cdot)$
with the $\lambda$-products defined by
\begin{equation}\label{bracket-pre-gd}
\widehat{a_\lambda b}=\partial(b\lhd a)+k(a\rhd b-b\lhd a)+\lambda(a\star b)+\alpha_\lambda(a,b)\mathfrak{c}_\beta, \;\;\text{for all $a$, $b\in V$,}
\end{equation}
where $\alpha_\lambda(a,b)\in\mathbb{C}\left[\lambda\right] $. Then for all $a,b\in V$, $\alpha_\lambda(a,b)=\sum\limits_{i=0}^{3}\lambda^i\alpha_i(a,b) $ ,
where all $\alpha_i(\cdot, \cdot) :$  $ V\times V\rightarrow \mathbb{C}$ are bilinear forms satisfying \eqref{qqlsa}-\eqref{qlsa9} where  $a\circ b=k(a\rhd b-b\lhd a)$ for all $a$, $b\in V$.

\delete{\begin{eqnarray}
&&\alpha_2(a\ast b,c)-\alpha_2(a,c\lhd b)=k[\alpha_3(b*a-a*b,c)+\alpha_3(a,b\rhd c)]+(\beta-k)\alpha_3(a,c\lhd b),\label{b1}\\
&&2\alpha_2(a\ast b,c)-\alpha_2(b\ast a,c)-\alpha_2(a,b\star c)=3k[\alpha_3(b*a-a*b,c)+\alpha_3(a,b\rhd c)],\label{b2}\\
&&\alpha_1(a\ast b,c)-\alpha_1(a,c\lhd b)=k[\alpha_2(b*a-a*b,c)+\alpha_2(a,b\rhd c)]+(\beta-k)\alpha_2(a,c\lhd b),\label{b3}\\
&&\alpha_1(a\ast b,c)-\alpha_1(b\ast a,c)-\alpha_1(a,b\star c)+\alpha_1(b,a\star c)=2k\alpha_2(b*a-a*b,c),\label{b4}\\
&&\alpha_0(a\ast b,c)-\alpha_0(a,c\lhd b)+\alpha_0(b,a\star c)=k[\alpha_1(b*a-a*b,c)+\alpha_1(a,b\rhd c)]\label{b5}\\
&&+(\beta-k)\alpha_1(a,c\lhd b),\nonumber\\
&&k[\alpha_0(a*b-b*a,c)+\alpha_0(b,a\rhd c)-\alpha_0(a,b\rhd c)]=(\beta-k)[\alpha_0(a,c\lhd b)-\alpha_0(b,c\lhd a)].\label{b6}
\end{eqnarray}
for all $a,b,c\in V$ and $k\in\mathbb{C}$. Particularly, when $\beta=0$, \eqref{b1}-\eqref{b6} are identical with
\begin{align}
&\alpha_2(a\ast b,c)-\alpha_2(a,c\lhd b)=k[\alpha_3(b*a-a*b,c)+\alpha_3(a,b\rhd c-c\lhd b)],\label{b0001}\\
&2\alpha_2(a\ast b,c)-\alpha_2(b\ast a,c)-\alpha_2(a,b\star c)=3k[\alpha_3(b*a-a*b,c)+\alpha_3(a,b\rhd c)],\label{b0002}\\
&\alpha_1(a\ast b,c)-\alpha_1(a,c\lhd b)=k[\alpha_2(b*a-a*b,c)+\alpha_2(a,b\rhd c-c\lhd b)],\label{b0003}\\
&\alpha_1(a\ast b,c)-\alpha_1(b\ast a,c)-\alpha_1(a,b\star c)+\alpha_1(b,a\star c)=2k\alpha_2(b*a-a*b,c),\label{b0004}\\
&\alpha_0(a\ast b,c)-\alpha_0(a,c\lhd b)+\alpha_0(b,a\star c)=k[\alpha_1(b*a-a*b,c)+\alpha_1(a,b\rhd c-c\lhd b)],\label{b0005}\\
&k[\alpha_0(a*b-b*a,c)+\alpha_0(b,a\rhd c-c\lhd a)-\alpha_0(a,b\rhd c-c\lhd b)]=0.\label{b0006}
\end{align}
for all $a,b,c\in V$ and $k\in\mathbb{C}$. Moreover, $\alpha_\lambda(\cdot,\cdot)$ is equivalent to $\alpha'_\lambda(\cdot,\cdot)$ if and only if $\alpha_i(a,b)=\alpha'_i(a,b)$ for $a,b\in V$, $i=2,3$, and
there is a linear map $\phi:$ $V\rightarrow\mathbb{C}$ such that
\begin{align}
&\alpha_1(a,b)=\alpha'_1(a,b)+\phi(a\star b),\\
&\alpha_0(a,b)=\alpha'_0(a,b)+\beta\phi(b\lhd a)+\phi(a\rhd b-b\lhd a).
\end{align}}

\end{cor}
\begin{proof}
In this case, (\ref{qls1}) becomes
\begin{equation}
\begin{aligned}
&-\mu\alpha_{\lambda+\mu}(b\ast a,c)+\lambda\alpha_{\lambda+\mu}(a\ast b,c)+k\alpha_{\lambda+\mu}(a\ast b-b\ast a,c)\\
&-(\lambda+\beta)\alpha_{\lambda}(a,c\lhd b)-\mu\alpha_{\lambda}(a,b\star c)-k\alpha_{\lambda}(a,b\rhd c-c\lhd b)\\
&=-(\mu+\beta)\alpha_{\mu}(b,c\lhd a)-\lambda\alpha_{\mu}(b,a\star c)-k\alpha_{\mu}(b,a\rhd c-c\lhd a).\label{qqls}
\end{aligned}
\end{equation}
Set $\alpha_\lambda(a,b)=\sum_{i=0}^{n_{a,b}}\lambda^i\alpha_i(a,b)$ for any $a$, $b\in V$, where $n_{a,b}$ is a nonnegative integer depending on $a$ and $b$.
In this case,
for fixed $a$, $b$, $c\in V$, we can assume that there exists an upper bound for the degrees of all
$\alpha_\lambda(\cdot,\cdot)$ appearing in (\ref{qqls}). Therefore, we assume that $\alpha_\lambda(\cdot,\cdot)=\sum_{i=0}^n\lambda^i\alpha_i(\cdot, \cdot)$ for all $\alpha_\lambda(\cdot,\cdot)$ appearing in (\ref{qqls}).
Then similar to the discussion given in the proof of Theorem \ref{lsndt1}, we have
$\alpha_n(a\ast b,c)=\alpha_n(b\ast a,c)=\alpha_n(a, b\star c)=\alpha_n(b, a\star c)=\alpha_n(a, b\rhd c)=\alpha_n(a,c\lhd b)=\alpha_n(b, c\lhd a)=\alpha_n(b, a\rhd c)=0$ if $n>3$. Then we also get $\alpha_m(a\ast b,c)=\alpha_m(b\ast a,c)=\alpha_m(a, b\star c)=\alpha_m(b, a\star c)=\alpha_m(a, b\rhd c)=\alpha_m(a,c\lhd b)=\alpha_m(b, c\lhd a)=\alpha_m(b, a\rhd c)=0$ for all $n\geq m>3$ by repeating this process.
Therefore, $\alpha_m(a\ast b,c)= \alpha_m(a, b\star c)=\alpha_m(a,b\lhd c)=\alpha_m(a,b\rhd c)=0$ for all $a$, $b$, $c\in V$ and $m>3$. Then by $V=V\ast V$ or $V=V\star V$ or $V=V\lhd V$ or $V=V\rhd V$, we can assume that
$\alpha_\lambda(a,b)=\sum_{i=0}^3\lambda^i\alpha_i(a,b)$. Then it follows directly from Theorem \ref{lsndt1}.
\end{proof}

\begin{rmk}
Note that Corollary \ref{corollaryb} also holds when $V$ is infinite-dimensional.
\end{rmk}
\begin{cor}\label{corollaryb0}
Let $(V,\lhd,\rhd)$ be a pre-Novikov algebra with $V=V\ast V$ or $V=V\star V$ or $V=V\lhd V$ or $V=V\rhd V$. Let $R=\mathbb{C}\left[\partial\right]V$ be the
corresponding quadratic left-symmetric conformal algebra and $\widehat{R}=R\oplus \mathbb{C}\mathfrak{c}_\beta$ be a central extension of $(R,\cdot_\lambda\cdot)$
with the $\lambda$-products defined by
\begin{equation}\label{bracket-pre-Novikov}
\widehat{a_\lambda b}=\partial(b\lhd a)+\lambda(a\star b)+\alpha_\lambda(a,b)\mathfrak{c}_\beta,\;\; \text{for all $a$, $b\in V$,}
\end{equation}
where $\alpha_\lambda(a,b)\in\mathbb{C}\left[\lambda\right] $. Then for all $a,b\in V$, $\alpha_\lambda(a,b)=\sum\limits_{i=0}^{3}\lambda^i\alpha_i(a,b) $ ,
where all $\alpha_i(\cdot, \cdot) :$  $ V\times V\rightarrow \mathbb{C}$ are bilinear forms satisfying \eqref{qqlsa} and
\begin{eqnarray}
&&\alpha_2(a\ast b,c)-\alpha_2(a,c\lhd b)-\beta\alpha_3(a,c\lhd b)=0,\label{b01}\\
&&2\alpha_2(a\ast b,c)-\alpha_2(b\ast a,c)-\alpha_2(a,b\star c)=0,\label{b02}\\
&&\alpha_1(a\ast b,c)-\alpha_1(a,c\lhd b)- \beta\alpha_2(a,c\lhd b)=0,\label{b03}\\
&&\alpha_1(a\ast b,c)-\alpha_1(b\ast a,c)-\alpha_1(a,b\star c)+\alpha_1(b,a\star c)=0,\label{b04}\\
&&\alpha_0(a\ast b,c)-\alpha_0(a,c\lhd b)+\alpha_0(b,a\star c)-\beta\alpha_1(a,c\lhd b)=0,\label{b05}\\
&&\beta\alpha_0(a,c\lhd b)=\beta\alpha_0(b,c\lhd a),\label{b06}
\end{eqnarray}
for all $a,b,c\in V$. In particualr, when $\beta=0$, \eqref{b01}-\eqref{b06} are identical with
\begin{align}
&\alpha_2(a\ast b,c)=\alpha_2(a,c\lhd b),\label{b001}\\
&\alpha_2(a\ast b,c)=\alpha_2(b\ast a,c)+\alpha_2(a,b\rhd c),\label{b002}\\
&\alpha_1(a\ast b,c)=\alpha_1(a,c\lhd b),\label{b003}\\
&\alpha_1(a,b\rhd c)=\alpha_1(b,a\rhd c),\label{b004}\\
&\alpha_0(a\ast b,c)-\alpha_0(a,c\lhd b)+\alpha_0(b,a\star c)=0,\label{b005}
\end{align}
for all $a,b,c\in V$. Moreover, $\alpha_\lambda(\cdot,\cdot)$ is equivalent to $\alpha'_\lambda(\cdot,\cdot)$ if and only if $\alpha_i(\cdot,\cdot)=\alpha'_i(\cdot,\cdot)$ for $i=2,3$, and
there is a linear map $\phi:$ $V\rightarrow\mathbb{C}$ such that
\begin{align}
&\alpha_1(a,b)=\alpha'_1(a,b)+\phi(a\star b),\\
&\alpha_0(a,b)=\alpha'_0(a,b)+\beta\phi(b\lhd a),\;\;\text{ for all $a, b \in V$.}
\end{align}
\end{cor}

\begin{proof}
It is straightforward from Corollary \ref{corollaryb} when $k=0$.
\end{proof}
\begin{cor}
Let $(V, \lhd, \rhd)$ be a pre-Novikov algebra and $R=\mathbb{C}[\partial]V$ be the corresponding quadratic left-symmetric conformal algebra. If there exists an element $e\in V$ such that
$a\ast e=a$ and $a\lhd e=a$ for all $a\in V$, then $H^2(R,\mathbb{C}_\beta)=0$ with $\beta \neq 0$.
\end{cor}

\begin{proof}
Obviously, $V\ast V=V$.
Let $b=e$ in (\ref{b01}), (\ref{b03}), (\ref{b05}) and (\ref{b06}). Then we get $\alpha_3(a, c)=\alpha_2(a, c)=0$, $\alpha_1(a, c)=\frac{1}{\beta}\alpha_0(e, b\star c)$ and $\alpha_0(a, c)=\alpha_0(e, c\lhd a)$ for all $a$, $c\in V$. Therefore, $\alpha_3(\cdot, \cdot)=\alpha_2(\cdot, \cdot)=0$. Let $\phi(a)=\frac{1}{\beta}\alpha_0(e,a)$ in Corollary \ref{corollaryb0}. Then we can make $\alpha_1(\cdot, \cdot)$ and $\alpha_0(\cdot, \cdot)$ be zero. Therefore, $H^2(R,\mathbb{C}_\beta)=0$.
\end{proof}
\delete{
\begin{cor}
Let $(V,\cdot)$ be a commutative associative algebra with $V=V\cdot V$ and $R=\mathbb{C}\left[\partial\right]V$ be the
corresponding quadratic left-symmetric conformal algebra. Suppose $\widehat{R}=R\oplus \mathbb{C}\mathfrak{c}_\beta$ be a central extension of $(R,\cdot_\lambda\cdot)$
with the $\lambda$-product defined by
\begin{equation}
\widehat{a_\lambda b}=\partial(a\cdot b)+\lambda(a\cdot b)+\alpha_\lambda(a,b)\mathfrak{c}_\beta,
\end{equation}
where $a,b\in V$ and $\alpha_\lambda(a,b)\in\mathbb{C}\left[\partial\right] $. Then for all $a,b\in V$, $\alpha_\lambda(a,b)=\sum\limits_{i=0}^{2}\lambda^i\alpha_i(a,b)$ ,
where all $\alpha_i(\cdot, \cdot) :$  $ V\times V\rightarrow \mathbb{C}$ are bilinear forms satisfying
\begin{eqnarray}
&&\alpha_2(a\cdot b,c)=\alpha_2(a,c\cdot b),\label{ccorl2}\\
&&\alpha_1(a,c\cdot b)+ \beta\alpha_2(a,c\cdot b)=\alpha_1(a\cdot b,c),\label{ccorl3}\\
&&\alpha_1(a,c\cdot b)=\alpha_1(b,c\cdot a),\label{ccorl4}\\
&&\alpha_0(a,c\cdot b)+\beta\alpha_1(a,c\cdot b)=\alpha_0(a\cdot b,c)
+\alpha_0(b,c\cdot a),\label{ccorl5}\\
&&\beta\alpha_0(a,c\cdot b))=\beta\alpha_0(b,c\cdot a),\label{ccorl6}
\end{eqnarray}
for any $a,b,c\in V$.  Moreover, $\alpha_\lambda(\cdot,\cdot)$ is equivalent to $\alpha'_\lambda(\cdot,\cdot)$ if and only if $\alpha_2(\cdot,\cdot)=\alpha'_2(\cdot,\cdot)$ for $a,b\in V$ and
there is a linear map $\phi:$ $V\rightarrow\mathbb{C}$ such that
\begin{align}
&\alpha_1(a,b)=\alpha'_1(a,b)+\phi(a\cdot b),\\
&\alpha_0(a,b)=\alpha'_0(a,b)+\beta\phi( a \cdot b).
\end{align}
\end{cor}

\begin{proof}
It is straightforward from Corollary \ref{corollaryb0}.
\end{proof}
}

\begin{pro}
Let $(V, \lhd, \rhd)$ be a pre-Novikov algebra and $R=\mathbb{C}[\partial]V$ be the corresponding quadratic left-symmetric conformal algebra. Let $\alpha_i(\cdot, \cdot)$ $(i=0, 1, 2, 3)$ be bilinear forms satisfying
(\ref{qqlsa}), (\ref{b001})-(\ref{b005}) for all $a$, $b$, $c\in V$. Set $\eta_i(\cdot, \cdot): \text{Coeff} (R)\times \text{Coeff} (R)\rightarrow \mathbb{C}$ be bilinear forms on $\text{Coeff} (R)$ as follows:
\begin{eqnarray*}
&&\eta_0(a\otimes t^m, b\otimes t^n)=\alpha_0(a,b)\delta_{m+n+1,0},\\
&&\eta_1(a\otimes t^m, b\otimes t^n)=m\alpha_1(a,b)\delta_{m+n,0},\\
&&\eta_2(a\otimes t^m, b\otimes t^n)=m(m-1)\alpha_2(a,b)\delta_{m+n-1,0},\\
&&\eta_3(a\otimes t^m, b\otimes t^n)=m(m-1)(m-2)\alpha_3(a,b)\delta_{m+n-2,0}.
\end{eqnarray*}
for all $a$, $b\in V$ and $m$, $n\in \mathbb{Z}$. Then $\eta_i$ $(i=0, 1, 2, 3)$ are 2-cocycles of the left-symmetric algebra $\text{Coeff} (R)$.
\end{pro}
\begin{proof}
Let $\alpha_\lambda(a,b)=\sum_{i=0}^3\lambda^i\alpha_i(a,b)$ for all $a$, $b\in V$. Then by Corollary \ref{corollaryb0}, there is a central extension $\widehat{R}$ of $R$ by a one-dimensional center $\mathbb{C}\mathfrak{c}_0$ given by (\ref{bracket-pre-Novikov}). Denote $\mathfrak{c}_0$ by $\mathfrak{c}$. Then the coefficient algebra
$\text{Coeff}(\widehat{R})$ is $\text{Coeff}(R)\oplus \mathbb{C}\mathfrak{c}\otimes t^{-1}$ with the non-trivial products given by
\begin{eqnarray*}
\widehat{(a\otimes t^m)\circ (b\otimes t^n)}&=&(a\circ b)\otimes t^{m+n}+m(a\rhd b)\otimes t^{m+n-1}-n(b\lhd a)\otimes t^{m+n-1}\\
&&+(\alpha_{0}(a,b)\delta_{m+n+1,0}+m\alpha_{1}(a,b)\delta_{m+n,0}+m(m-1)\alpha_2(a,b)\delta_{m+n-1,0}\\
&&+m(m-1)(m-2)\alpha_3(a,b)\delta_{m+n-2,0})\mathfrak{c} \otimes t^{-1}.
\end{eqnarray*}
Therefore, $\text{Coeff}(\widehat{R})$ is a central extension of $\text{Coeff}(R)$ by a one-dimensional center $\mathbb{C}\mathfrak{c}\otimes t^{-1}$. Note that $\alpha_i(\cdot, \cdot)$ $(i=0, 1, 2, 3)$ do not depend on each other. Therefore, by the cohomology theory of left-symmetric algebras given in \cite{D},  $\eta_i$ $(i=0, 1, 2, 3)$ are 2-cocycles of $\text{Coeff} (R)$.
\end{proof}

Next, we present several examples to compute $H^2(R, \mathbb{C})$ where $\mathbb{C}=\mathbb{C}_0$.
\begin{ex}\label{e1}
Let $R_c=\mathbb{C}[\partial]L$ be the left-symmetric conformal algebra given in Example \ref{rank one}. It is obvious that $R_c$ is the quadratic left-symmetric conformal algebra corresponding to a 1-dimensional left-symmetric Poisson algebra $(V=\mathbb{C}L, \cdot, \circ)$  defined as follows:
\begin{equation}
L\cdot L=L,~~L\circ L=cL.
\end{equation}
 Therefore, by
Corollary \ref{corleft}, \eqref{corl2}-\eqref{corl6} with $\beta=0$ and some simple computations, we can get\\
(1) If $c\neq0$,
\begin{eqnarray*}
\alpha_2(L,L)=0,~~\alpha_1(L,L)=A,~~\alpha_0(L,L)=cA,
\end{eqnarray*}
for any $A\in\mathbb{C}$.\\
(2) If $c=0$,
\begin{eqnarray*}
\alpha_3(L,L)=0,~~\alpha_2(L,L)=A,~~\alpha_1(L,L)=B,~~\alpha_0(L,L)=0,
\end{eqnarray*}
for any $A, B\in\mathbb{C}$.

This is the same as the result given in Proposition 3.4 in \cite{HL}.
If $c=0$, by choosing the linear map $\phi:$ $V\rightarrow\mathbb{C}$ in Corollary \ref{corleft} defined by $\phi(L)=B$, we
can make $\alpha_1(\cdot,\cdot)$ be zero up to equivalence. Thereby, by Corollary \ref{corleft}, all equivalence classes of central
extensions of $R_0$ by a one-dimensional centre $\mathbb{C}\mathfrak{c}$ are $\widehat{R_0}(A)$ with the non-trival $\lambda$-product given as follows:
$$
L_\lambda L=(\partial+\lambda)L+A\lambda^2\mathfrak{c},
$$
for all $A\in\mathbb{C}$. Note that if $A_1\neq A_2$ , then $\widehat{R_0}(A_1)$ is not equivalent to $\widehat{R_0}(A_2)$.
Therefore, in this case, $\text{dim} H^2( R_0,\mathbb{C})=1  $. Similarly, we can get $\text{dim} H^2( R_c,\mathbb{C})=0$ when $c\neq 0$.
\end{ex}

\begin{ex}
Let $R=\mathbb{C}\left[\partial\right]L \oplus \mathbb{C}\left[\partial\right]W$ be the left-symmetric conformal algebra with $\lambda$-products
given as follows:
\begin{align}
&L_\lambda L= (\partial+\lambda)L,~~L_\lambda W=(\partial+2\lambda )W,\\
&W_\lambda W=0,~~W_\lambda L=0,
\end{align}
Next, we compute $H^2(R, \mathbb{C})$.

Actually, $R$ is the quadratic left-symmetric conformal algebra corresponding to a 2-dimensional pre-Gel'fand-Dorfman algebra $(V=\mathbb{C}L\oplus \mathbb{C}W, \lhd, \rhd, \circ)$ given by
\begin{equation*}
\begin{aligned}
&L\lhd L=L,~~L\circ L=0,~~L\rhd L=0 ,\\
&W\lhd L=W,~~L\circ W=0,~~L\rhd W=W,\\
&L\lhd W=0,~~W\circ L=0,~~W\rhd L=0,\\
&W\lhd W=0,~~W\circ W=0,~~W\rhd W=0.
\end{aligned}
\end{equation*}
Obviously, $V \lhd V=V$. Therefore, by
Corollary \ref{lsndc1}, and by some simple computations, we obtain
\begin{eqnarray*}
&&\alpha_3(L,W)=A,\alpha_2(L,L)=B,\alpha_1(L,L)=C,\alpha_1(L,W)=D,\\
&&\alpha_3(L,L)=\alpha_3(W,L)=\alpha_3(W,W)=\alpha_2(L,W)=\alpha_2(W,L)=\alpha_2(W,W)\\
&=&\alpha_1(W,L)=\alpha_0(L,L)=\alpha_0(W,L)=\alpha_0(L,W)=\alpha_1(W,W)=\alpha_0(W,W)=0,
\end{eqnarray*}
for any $A,B,C,D,\in\mathbb{C}$.

\delete{Actually, if $a\neq0,a\neq2,a\neq1$, or $a\neq0,b\neq0$, we can choose the linear map $\phi:$ $V\rightarrow\mathbb{C}$ in Corollary \ref{lsndc1}
defined by $\phi(L)=B$ and $\phi(W)=\frac{C}{a}$ to make $\alpha_1(\cdot,\cdot)$ and $\alpha_0(\cdot,\cdot)$ be zero up to an equivalence.
Thereby, by Corollary \ref{lsndc1}, all equivalence classes of central extensions of R(a,b) by a one-dimensional centre $\mathbb{C}\mathfrak{c}$ are $\widehat{R(a,b)}(A)$
with the $\lambda$-product as follows:
$$
\begin{aligned}
&L_\lambda L=(\partial+\lambda)L+A\lambda^2\mathfrak{c},~~L_\lambda W=(\partial+a\lambda+b )W,\\
&W_\lambda W=0,~~W_\lambda L=0,
\end{aligned}
$$
for all $A\in\mathbb{C}$. Note that if $A_1\neq A_2$ , then $\widehat{R(a,b)}(A_1)$ is not identical with $\widehat{R(a,b)}(A_2)$.
Therefore, in this case, dim$H^2( R(a,b),\mathbb{C}\mathfrak{c})=1  $ .\par}

Choose the linear map $\phi:$ $V\rightarrow\mathbb{C}$ in Corollary \ref{lsndc1}
defined by $\phi(L)=C$ and $\phi(W)=\frac{D}{2}$ to make $\alpha_1(\cdot,\cdot)$ and $\alpha_0(\cdot,\cdot)$ be zero up to  equivalence.
Therefore, by Corollary \ref{lsndc1}, all equivalence classes of central extensions of $R$ by a one-dimensional centre $\mathbb{C}\mathfrak{c}$ are $\widehat{R}(A,B)$
with the $\lambda$-products as follows:
$$
\begin{aligned}
&L_\lambda L=(\partial+\lambda)L+B\lambda^2\mathfrak{c},~~L_\lambda W=(\partial+2\lambda)W+A\lambda^3\mathfrak{c},\\
&W_\lambda W=0,~~W_\lambda L=0,
\end{aligned}
$$
for all $A,B\in\mathbb{C}$. Note that if $(A_1,B_1)\neq (A_2,B_2)$ , then $\widehat{R}(A_1,B_1)$ is not equivalent to $\widehat{R}(A_2,B_2)$.
Therefore, in this case, \text{dim}$H^2( R,\mathbb{C})=2 $ .

\delete{
If $a=1$ and $b=0$, we can choose the linear map $\phi:$ $V\rightarrow\mathbb{C}$ in Corollary \ref{lsndc1}
defined by $\phi(L)=C$ and $\phi(W)=D$ to make $\alpha_1(L,W),\alpha_1(L,L)$ and $\alpha_0(\cdot,\cdot)$ be zero up to an equivalence.
Thereby, by Corollary \ref{lsndc1}, all equivalence classes of central extensions of R(1,0) by a one-dimensional centre $\mathbb{C}\mathfrak{c}$ are $\widehat{R(1,0)}(A)$
with the $\lambda$-product as follows:
$$
\begin{aligned}
&L_\lambda L=(\partial+\lambda)L+A\lambda^2\mathfrak{c},~~L_\lambda W=(\partial+2\lambda)W+B\lambda^2\mathfrak{c},\\
&W_\lambda W=E\lambda\mathfrak{c},~~W_\lambda L=F\lambda\mathfrak{c},
\end{aligned}
$$
for all $A,B,E,F\in\mathbb{C}$. Note that if $(A_1,B_1,E_1,F_1)\neq (A_2,B_2,E_2,F_2)$ , then $\widehat{R(1,0)}(A_1,B_1,E_1,F_1)$ is not identical
with $\widehat{R(1,0)}(A_2,B_2,E_2,F_2)$.
Therefore, in this case, dim$H^2( R(1,0),\mathbb{C}\mathfrak{c})=4 $ .\par
If $a=0$ and $b=0$, we can choose the linear map $\phi:$ $V\rightarrow\mathbb{C}$ in Corollary \ref{lsndc1}
defined by $\phi(L)=B$ to make $\alpha_1(W,L),\alpha_1(L,L),\alpha_1(W,W),\alpha_0(L,L)\alpha_0(W,W)$ and $\alpha_0(W,L)$ be zero up to an equivalence.
Thereby, by Corollary \ref{lsndc1}, all equivalence classes of central extensions of R(0,0) by a one-dimensional centre $\mathbb{C}\mathfrak{c}$ are $\widehat{R(0,0)}(A)$
with the $\lambda$-product as follows:
$$
\begin{aligned}
&L_\lambda L=(\partial+\lambda)L+A\lambda^2\mathfrak{c},~~L_\lambda W=\partial W+(C\lambda+D)\mathfrak{c},\\
&W_\lambda W=0,~~W_\lambda L=0,
\end{aligned}
$$
for all $A,C,D\in\mathbb{C}$. Note that if $(A_1,C_1,D_1)\neq (A_2,C_2,D_2)$ , then $\widehat{R(0,0)}(A_1,C_1,D_1)$ is not identical
with $\widehat{R(0,0)}(A_2,C_2,D_2)$.
Therefore, in this case, dim$H^2( R(0,0),\mathbb{C}\mathfrak{c})=3 $ .\par
If $a=0$ and $b=0$, we can choose the linear map $\phi:$ $V\rightarrow\mathbb{C}$ in Corollary \ref{lsndc1}
defined by $\phi(L)=B$ and $\phi(W)=\frac{C}{b}$ to make $\alpha_1(\cdot,\cdot)$ and $\alpha_0(\cdot,\cdot)$ be zero up to an equivalence.
Thereby, by Corollary \ref{lsndc1}, all equivalence classes of central extensions of R(0,b) by a one-dimensional centre $\mathbb{C}\mathfrak{c}$ are $\widehat{R(0,b)}(A)$
with the $\lambda$-product as follows:
$$
\begin{aligned}
&L_\lambda L=(\partial+\lambda)L+A\lambda^2\mathfrak{c},~~L_\lambda W=(\partial+b) W,\\
&W_\lambda W=0,~~W_\lambda L=0,
\end{aligned}
$$
for all $A\in\mathbb{C}$. Note that if $A_1\neq A_2$ , then $\widehat{R(0,b)}(A_1)$ is not identical
with $\widehat{R(0,b)}(A_2)$.
Therefore, in this case, dim$H^2( R(0,b),\mathbb{C}\mathfrak{c})=1 $ .}
\end{ex}

\begin{ex}
Let $R_1=\bigoplus_{i\in \ZZ}\mathbb{C}\left[\partial\right]x_i$ be the infinite left-symmetric conformal algebra with $\lambda$-products given as follows:
\begin{equation*}
{x_i}_\lambda x_j=(\partial +\lambda+1)x_{i+j},\;\;\text{for all}~~i,j\in\mathbb{Z}.
\end{equation*}
Next, we compute $H^2(R_1, \mathbb{C})$.

It is clear that $R_1$ is the quadratic left-symmetric conformal algebra corresponding to an infinite-dimensional pre-Gel'fand-Dorfman algebra $V=\bigoplus_{i\in\mathbb{Z}}\mathbb{C}x_i$
with operations `` $\lhd$ ", ``$\rhd$" and `` $\circ$ "  given as follows:
\begin{eqnarray*}
&&x_i\lhd x_j=x_{i+j},\;\; x_i\rhd x_j=0,\\
&&x_i\circ x_j=x_{i+j},\;\; \text{for all $i$, $j\in \ZZ$.}
\end{eqnarray*}
By Corollary \ref{corollaryb} and some simple computations, we get
\begin{align}
&\alpha_3(x_i,x_j) = 0,\\
&\alpha_2(x_i,x_{j+k}) =\alpha_2(x_{i+j},x_k)=\alpha_2(x_j,x_{i+k}) ,\label{al1}\\
&\alpha_2(x_{i+j},x_k)=\alpha_1(x_{i+j},x_k)-\alpha_1(x_i,x_{j+k}),\label{al2}\\
&\alpha_1(x_j,x_{i+k})=\alpha_1(x_i,x_{j+k}),\label{al3}\\
&\alpha_1(x_i,x_{j+k})=\alpha_0(x_{i+j},x_k)+\alpha_0(x_j,x_{i+k})-\alpha_0(x_i,x_{j+k}),\label{al4}\\
&\alpha_0(x_j,x_{i+k})=\alpha_0(x_i,x_{j+k})\label{al5},
\end{align}
for any $i,j,k\in \mathbb{Z}$. By \eqref{al1}, we get $\alpha_2(x_i,x_{j+k})=\alpha_2(x_{i+j+k},x_0)=\alpha_2(x_0,x_{i+j+k}) $. Thus, we set
$\alpha_2(x_i,x_{j+k})=\alpha_2(x_0,x_{i+j+k})=f(i+j+k)$, for some complex function $f$. Moreover, we get $ \alpha_2(x_i,x_j)=f(i+j)$ by letting $k=0$. However, by letting $j=0$ in \eqref{al2}, we get $f(i+k)=0 $. Therefore $f(i)=0$ for all $i\in\mathbb{Z}$. Similarly, we set $\alpha_1(x_i,x_0)=\alpha_1(x_0,x_i)=g(i)$ for some complex function $g$. Moreover, we have $\alpha_1(x_{i+j},x_k)=\alpha_1(x_i,x_{j+k})=\alpha_0(x_{i+j},x_k)$ by \eqref{al4} and \eqref{al5}. Thus $\alpha_0(x_i,x_j)=\alpha_1(x_i,x_j)=g(i+j)$.  By choosing the linear map $\phi$: $ V\rightarrow\mathbb{C}$ in Theorem \ref{lsndt1} defined by $\phi(x_i)=g(i)$ for all $i\in\mathbb{Z}$, we can make $\alpha_1(\cdot,\cdot)$ and $\alpha_0(\cdot,\cdot)$ be zero up to equivalence.
Consequently, all equivalence classes of central extensions of $R$ by a one-dimensional centre
$\mathbb{C}\mathfrak{c}$ are $\widehat{R}(f)=R\oplus \mathbb{C}\mathfrak{c}$ with the $\lambda$-products as follows:
$$
{x_i}_\lambda x_j=(\partial+\lambda+1)x_{i+j},
$$
for
all $i,j\in\mathbb{Z} $. Therefore, \text{dim} $H^2(R, \mathbb{C})=0$.

Let $R_2=\bigoplus_{i\in \ZZ}\mathbb{C}\left[\partial\right]x_i$ be the infinite left-symmetric conformal algebra with $\lambda$-products given as follows:
\begin{equation*}
{x_i}_\lambda x_j=(\partial +\lambda)x_{i+j},\;\;\text{for all}~~i,j\in\mathbb{Z}.
\end{equation*}
By Corollary \ref{corollaryb} and some simple computations, we get
\begin{align}
&\alpha_3(x_i,x_j) = 0,\\
&\alpha_2(x_i,x_{j+k}) =\alpha_2(x_{i+j},x_k)=\alpha_2(x_j,x_{i+k}) ,\label{alpha2}\\
&\alpha_1(x_i,x_{j+k})=\alpha_1(x_{i+j},x_k)=\alpha_1(x_j,x_{i+k}),\label{alpha1}\\
&\alpha_0(x_i,x_{j+k})=\alpha_0(x_{i+j},x_k)+\alpha_0(x_j,x_{i+k}),\label{alpha0}
\end{align}
for any $i,j,k\in \mathbb{Z}$.  Similar to the discussion above, we get that  $ \alpha_2(x_i,x_j)=f(i+j)$, $\alpha_1(x_i,x_j)=g(i+j) $ and $\alpha_0(x_i,x_j)=0$ for all $i$, $j\in \ZZ$ and some complex functions $f$ and $g$.

\delete{By \eqref{alpha2}, we get $\alpha_2(x_i,x_{j+k})=\alpha_2(x_{i+j+k},x_0)=\alpha_2(x_0,x_{i+j+k}) $. Thus, we set
$\alpha_2(x_i,x_{j+k})=\alpha_2(x_0,x_{i+j+k})=f(i+j+k)$, for some complex function $f$. Moreover, we get $ \alpha_2(x_i,x_j)=f(i+j)$ by letting $k=0$.
Similarly, $\alpha_1(x_i,x_j)=g(i+j) $ for some complex function $g$. Then, we consider $\alpha_0(x_i,x_j) $. Set $j=0$ in \eqref{alpha0}. We have
$\alpha_0(x_0,x_{i+k})=0 $ for any $i,k\in\mathbb{Z}$. Hence, it is equal to say $ \alpha_0(x_0,x_i)=0$ for all $i\in\mathbb{Z}$. Then we get
$2\alpha_0(x_j,x_k)=0 $ for any $j,k\in\mathbb{Z}$ by letting $i=0$ in \eqref{alpha0}. Consequently, $\alpha_0(x_i,x_j)=0 $ for any $i,j,k\in\mathbb{Z}$.\par}
Therefore, by choosing the linear map $\phi$: $ V\rightarrow\mathbb{C}$ in Corollary \ref{corollaryb0} defined by $\phi(x_i)=g(i)$ for all $i\in\mathbb{Z}$, we can make $\alpha_1(\cdot,\cdot)$ be zero up to  equivalence. Consequently, all equivalence classes of central extensions of $R_2$ by a one-dimensional centre
$\mathbb{C}\mathfrak{c}$ are $\widehat{R_2}(f)=R_2\oplus \mathbb{C}\mathfrak{c}$ with the $\lambda$-products as follows:
$$
{x_i}_\lambda x_j=(\partial+\lambda)x_{i+j}+f(i+j)\lambda^2\mathfrak{c},
$$
for all $i,j\in\mathbb{Z}$ and all complex function $f$. Note that if $f_1\neq f_2$, then $\widehat{R_2}(f_1) $ is not equivalent to $\widehat{R_2}(f_2) $. Therefore,
\text{dim} $H^2(R_2, \mathbb{C})=\infty$.

\end{ex}

\section{Simplicities of quadratic left-symmetric conformal algebras}
In this section, we will investigate the simplicities of quadratic left-symmetrical conformal algebras.

For a current left-symmetric conformal algebra $R=\mathbb{C}[\partial]L$ associated with a left-symmetric algebra $L$, it is easy to see that $R=\mathbb{C}[\partial]L$ is simple if and only if $L$ is simple. Therefore, in this section, we always assume that quadratic
 left-symmetric conformal algebras are not current in the sequel.
\delete{ To begin with, let us define $R$ is a quadratic left-symmetric conformal algebra corresponding to $(V,\lhd,\rhd,\circ )$ defined by Theorem \ref{lspn} , then the $\lambda$-product of a quadratic left-symmetric
 conformal algebra $R=\mathbb{C}[\partial]V$ given as follows
\begin{equation}
a_\lambda b=\partial(b\lhd a)+a\circ b+\lambda(a\rhd b+b\lhd a),~~\forall a,b\in V. \label{S1}
\end{equation}}

\begin{defi}
A subspace $I$ of a pre-Novikov algebra $(V,\lhd,\rhd)$ is called an {\bf ideal} of $V$ if
$a\lhd b, b\lhd a,a\rhd b, b\rhd a\in I $ for all $a\in I$ and $b\in V$.

Any nonzero pre-Novikov algebra $V$ has
two trivial ideals 0 and $V$. An ideal $I$ of $(V,\lhd,\rhd)$ is called {\bf proper}, if $I$ is not trivial.
A pre-Novikov algebra $(V,\lhd,\rhd)$ is called {\bf simple} if $V$ is non-trivial and has no proper ideals.

An ideal $I$ in the pre-Gel'fand-Dorfman algebra $(V, \lhd, \rhd, \circ)$
is called a {\bf proper ideal} if $I$ is both a proper ideal of $(V, \lhd, \rhd)$ and a proper ideal of $(V, \circ)$. A pre-Gel'fand-Dorfman algebra $(V,\lhd,\rhd, \circ)$ is called {\bf simple} if $V$ is non-trivial and has no proper ideals.  \end{defi}

Next, we give some necessary conditions for a quadratic left-symmetric conformal algebra to be simple.

\begin{pro}\label{simple1}
Let $R = \mathbb{C}\left[\partial\right]V$ be the quadratic left-symmetric conformal algebra corresponding to a pre-Gel'fand-Dorfman algebra $(V,\lhd,\rhd,\circ )$. If $R = \mathbb{C}\left[\partial\right]V$ is simple, then $(V,\lhd,\rhd,\circ )$ is simple.
\end{pro}

\begin{proof}
Suppose that $I$ is a proper ideal of $(V,\lhd,\rhd,\circ )$.
By (\ref{lnd}), we get that $\mathbb{C}\left[\partial\right]I$ is a proper ideal of $R = \mathbb{C}\left[\partial\right]V$. This contradicts with the simplicity of $R$. Thus this proposition holds.
\end{proof}

\delete{\begin{rmk}
Note that the necessity in Proposition \ref{simple1} is not sufficient. We will illustrate this by Example \ref{exsim} and then we will investigate under what circumstances it will be a simple quadratic left-symmetric conformal algebra.
\end{rmk}

\begin{ex}\label{exsim}
\delete{Given a two dimensional simple pre-Novikov algebra $(V,\triangleleft,\triangleright)=\mathbb{C}L\oplus\mathbb{C}W$ defined by:
\begin{align*}
&L\triangleleft L=-2L,~~~~L\triangleleft W=2L,~~~~W\triangleleft W=2W,~~~~W\triangleleft L=-2W,\\
&L\triangleright L=0,~~~~L\triangleright W=2L+2W,~~~~W\triangleright W=-2L-2W,~~~~W\triangleright L=0.
\end{align*}
Let `` $\circ$ " be a trivial operator in the pre-Gel'fand-Dorfman algebra $(V,\triangleleft,\triangleright,\circ)$. There is no doubt that $V$ is also simple and then the corresponding quadratic left symmetric conformal algebra $R=\mathbb{C}[\partial] V$ is not simple in that it has a propel ideal $I=\{a\partial^iL+b\partial^jW\mid a,b\in\mathbb{C},i\in\mathbb{N},j\in\mathbb{N}\setminus\{0\}\}$.}
\end{ex}
}
\begin{lem}\label{simlemma}
Let $(V,\lhd,\rhd)$ be a simple pre-Novikov algebra. Then $a\rhd b=-b\lhd a$ does not hold for all $a,b\in V$.
\end{lem}
\begin{proof}
If $a\rhd b=-b\lhd a$  for all $a,b\in V$, then by \eqref{ND3} and \eqref{ND4}, we have
$$
c\lhd(a\lhd b-b\lhd a)=(c\lhd a)\lhd b=(c\lhd b)\lhd a=c\lhd(b\lhd a-a\lhd b).
$$
Therefore $(c\lhd b)\lhd a=0$ for all $a,b,c\in V$. Thus, $(V\lhd V)\lhd V=0$. Note that in this case $V\lhd V$ is an ideal of  $(V,\lhd,\rhd)$. Since $(V,\lhd,\rhd)$ is simple, we
get that $V\lhd V=0$, which is impossible. Consequently, $a\rhd b=-b\lhd a$ does not hold for all $a,b\in V$.
\end{proof}

\begin{thm}\label{simpleth}
Let $(V,\lhd,\rhd,\circ )$ be a pre-Gel'fand-Dorfman algebra. Set $a\star b=
a\rhd b+b\lhd a$ for all $a$, $b\in V$. If $(V,\lhd,\rhd)$ is a simple pre-Novikov algebra, and
$V=V\star V$, then the quadratic left-symmetric conformal algebra $R=\mathbb{C}\left[\partial\right]V $ corresponding to $(V,\lhd,\rhd,\circ )$
is simple.
\end{thm}
\begin{proof}
Suppose that $I$ is a nonzero ideal of $R$ and $\beta=\sum\limits_{i=0}^{n}f_i(\partial)a_i\in I\setminus \{0\}$, where $a_i\in V(0\leq i\leq n)$ are
 linearly independent and $f_i(\partial)\in \mathbb{C}\left[\partial\right] \setminus \{0\}$. Assume that the degrees of $f_m(\partial)$, $\cdots$, $f_l(\partial)$
 are maximal in those $f_i(\partial)$, and the leading coefficients of $f_m(\partial)$, $\cdots$, $f_l(\partial)$ are $k_m$, $\cdots$, $k_l$. Let the degree of $f_m(\partial)$ be $j$. Suppose that $f_g(\partial)$, $\cdots$, $f_h(\partial) $ are the polynomials in those $f_i(\partial)$ whose degrees are $j-1$, and $k_g$, $\cdots$, $k_h$ are their leading coefficients respectively. For any $a\in V$,
 \begin{equation}
 \begin{aligned}
 &a_\lambda \beta=\sum\limits_{i=0}^{n}f_i(\partial+\lambda)(\partial(a_i\lhd a)+a\circ a_i+\lambda(a\star a_i)),\\
 &\beta_\lambda a=\sum\limits_{i=0}^{n}f_i(-\lambda)(\partial(a\lhd a_i)+a_i\circ a+\lambda(a_i\star a)).\label{lnds}
 \end{aligned}
 \end{equation}

By the coefficients of $\lambda^{j+1}$ in \eqref{lnds}, we get that $w_1=a\star t \in I$ and $w_2=t\star a\in I$ where $t=k_ma_m+\cdots+k_la_l$, for all $a\in V$. Let $U_1=\{t\}$ and $W=\{b|b\star a=a\star b=0 ~~\text{for all}~~a\in V\}$. Note that $t\neq 0$.

We claim that there exists some nonzero element $w\in V\cap I$. If either $w_1$ or $w_2$ is not zero, then we are done. If $a\star t=t\star a=0 $, i.e. $U_1\subseteq W$, then  by comparing the coefficients of $\lambda^{j}$ in \eqref{lnds}, we get that
$\partial(t\lhd a)+a\circ t+a\star c\in I$ and $\partial(a\lhd t)+t\circ a-c\star a \in I$ where $c=k_ga_g+\cdots+k_ha_h$, for all $a\in V$. Obviously, one of $t\lhd a$  and $a\lhd t$ is not zero since $(V,\lhd,\rhd)$ is simple.
Let $u=a\circ t+a\star c$ and $v=t\circ a-c\star a$. Then for any $b\in V$, we get
\begin{equation}\label{lnds2}
\begin{aligned}
b_\lambda (\partial (t\lhd a)+u)=&(\lambda+\partial)\big(\partial((t\lhd a)\lhd b)+b\circ (t\lhd a)+\lambda(b\star (t\lhd a))\big)\\
&+\partial(u\lhd b)+b\circ u+\lambda(b\star u),\\
b_\lambda (\partial (a\lhd t)+v)=&(\lambda+\partial)\big(\partial((a\lhd t)\lhd b)+b\circ (a\lhd t)+\lambda(b\star (a\lhd t))\big)\\
&+\partial(v\lhd b)+b\circ v+\lambda(b\star v),\\
(\partial (t\lhd a)+u)_\lambda b=&-\lambda\big(\partial(b\lhd (t\lhd a))+(t\lhd a)\circ b+\lambda( (t\lhd a)\star b)\big)\\
&+\partial(b\lhd u)+u\circ b+\lambda(u\star b),\\
(\partial (a\lhd t)+v)_\lambda b=&-\lambda\big(\partial(b\lhd (a\lhd t))+(a\lhd t)\circ b+\lambda( (a\lhd t)\star b)\big)\\
&+\partial(b\lhd v)+v\circ b+\lambda(v\star b).
\end{aligned}
\end{equation}
Thus, we get that $b\star (t\lhd a)$, $b\star (a\lhd t)$, $(t\lhd a)\star b$, $(a\lhd t)\star b \in I$  and
 $\partial((t\lhd a)\lhd b)+b\circ (t\lhd a)+b\star u $, $\partial((a\lhd t)\lhd b)+b\circ (a\lhd t)+b\star v$, $\partial(b\lhd (t\lhd a))+(t\lhd a)\circ b-u\star b$, $\partial(b\lhd (a\lhd t))+(a\lhd t)\circ b-v\star b\in I$ by comparing the coefficients of $\lambda^2$ and $\lambda$  in \eqref{lnds2} respectively. Then we get $b\star (t\rhd a) $, $b\star (a\rhd t)$, $(t\rhd a)\star b$, $(a\rhd t)\star b \in I$ by $a\star t=t\star a=0$. Let $U_2=\{t\rhd a,t\lhd a, a\rhd t, a\lhd t|~~\text{for all}~~a\in V\}$. If one of $b\star (t\lhd a)$, $b\star (a\lhd t)$, $(t\lhd a)\star b$, $(a\lhd t)\star b$, $b\star (t\rhd a) $, $b\star (a\rhd t)$, $(t\rhd a)\star b$ and $(a\rhd t)\star b$ is nonzero, we are done. Otherwise $U_2\subseteq W$. Then by repeating the above steps, we can obtain $U_3$, $U_4$, $\cdots$, $U_n$, $\cdots$. If there exists a nonzero element $w$ in some $U_n$ such that $w\star b\neq0$ or $b\star w\neq0$ for some $b\in V$, we are done. Otherwise all $U_i$ $\subseteq W$.
Therefore there is an ascending sequence
$$
U_1\subseteq U_1+U_2\subseteq U_1+U_2+U_3\subseteq\cdots\subseteq V.
$$
Since $(V, \lhd, \rhd)$ is simple, we get $W=V$. Therefore $a\lhd b=-b\rhd a $ for all $a,b\in V$. By Lemma \ref{simlemma}, it
is impossible. Consequently, the claim holds.

Suppose that $w$ is a nonzero element in $V\cap I$. Since for any $b\in V$,
\begin{equation*}
\begin{aligned}
&b_\lambda w=\partial(w\lhd b)+b\circ w+\lambda(b\star w),\\
&w_\lambda b=\partial(b\lhd w)+w\circ b+\lambda(w\star b),
\end{aligned}
\end{equation*}
we get that $\partial(w\lhd b)+b\circ w\in I$, $\partial(b\lhd w)+w\circ b \in I$,
$b\star w\in I$ and $w\star b\in I $ by the coefficients of $\lambda^0$ and $\lambda$ respectively. Therefore we obtain that
$\partial(b\rhd w)-b\circ w \in I$ and $\partial(w\rhd b)-w\circ b \in I$.
Let $U=\{w\lhd b,b\lhd w,b\rhd w,w\rhd b|~~\text{for all}~~ b\in V \}$.
We have $\partial u+v\in I$ for all $u\in U$, and some $v\in V$. By
\begin{align*}
b_\lambda (\partial u+v)=&(\lambda+\partial)(\partial(u\lhd b)+b\circ u+\lambda(b\star u))+\partial(v\lhd b)+b\circ v+\lambda(b\star v),\nonumber\\
(\partial u+v)_\lambda b=&-\lambda(\partial(b\lhd u)+u\circ b+\lambda(u\star b))+\partial(b\lhd v)+v\circ b+\lambda(v\star b),\nonumber
\end{align*}
we obtain that $b\star u \in I$, $u\star b\in I$,  $\partial(u\lhd b)+b\circ u+b\star v\in I$ and $\partial(b\lhd u)+u\circ b-v\star b \in I$
from the coefficients of $\lambda^2$ and $\lambda$ respectively. Then we get that $\partial(b\rhd u)-b\circ u-b\star v \in I$ and $\partial(u\rhd b)-u\circ b+v\star b \in I$. Therefore, by setting $H_1=\{u\lhd b,b\lhd u,b\rhd u,u\rhd b|~~\text{for all}~~ b\in V ,u\in U\}$  and proceeding inductively we have $H_2$, $\cdots$, $H_n$, $\cdots$ and an ascending sequence
$$
U\subseteq H_1+U\subseteq U+H_1+H_2\subseteq\cdots\subseteq V.
$$
Since $V$ is simple, we have $\partial v+c\in I$ for all $v\in V$ and some $c\in V$. Thus by
\begin{align*}
b_\lambda (\partial v+c)=&(\lambda+\partial)(\partial(v\lhd b)+b\circ v+\lambda(b\star v))+\partial(c\lhd b)+b\circ c+\lambda(b\star c),\nonumber
\end{align*}
we obtain $b\star v\in I$ from the coefficient of $\lambda^2$ for all $b,v\in V$.
Since $V\star V=V$ and $I$ is a $ \mathbb{C}\left[\partial\right] $-module, one can get $V\subset I$ and then $I=R$. Therefore, $R$ is simple.
\end{proof}

\begin{ex}
Let $R_c=\mathbb{C}[\partial]L$ be the left-symmetric conformal algebra given in Example \ref{rank one}. Note that the pre-Gel'fand-Dorfman Poisson algebra corresponding to  $(V=\mathbb{C}L, \lhd, \rhd, \circ)$  defined as follows:
\begin{equation}
L\lhd L=L,~~L\rhd L=0,~~L\circ L=cL.
\end{equation}
Obviously, $(V, \lhd, \rhd)$ is simple and $V=V\star V$. Therefore, by Theorem \ref{simpleth}, $R_c$ is simple for any $c\in \mathbb{C}$.
\end{ex}

In what follows, we investigate the simplicities of quadratic left-symmetric conformal algebras associated with pre-Gel'fand-Dorfman algebras $(V,\lhd,\rhd,\circ )$
with ``$\rhd$" trivial.

\begin{pro}\label{simcor}
Let $(V,\lhd,\rhd,\circ )$ be a simple pre-Gel'fand-Dorfman algebra with ``$\rhd$"  trivial.
If there exists an element $a\in V$ such that $a\lhd b\neq0$ or $b\lhd a\neq0 $ for all non-zero $b\in V$, then the quadratic left-symmetric conformal algebra $R=\mathbb{C}\left[\partial\right]V $ corresponding to $(V,\lhd,\rhd,\circ )$
is simple.

\end{pro}
\begin{proof}

With the assumption as the proof of Theorem \ref{simpleth},
for any $b\in V$, we have
 \begin{equation}
 \begin{aligned}
 &b_\lambda \beta=\sum\limits_{i=0}^{n}f_i(\partial+\lambda)(\partial(a_i\lhd b)+b\circ a_i+\lambda(a_i\lhd b)).
 \end{aligned}
 \end{equation}
Since there exists an element $a\in V$ such that $a\lhd b\neq0$ or $b\lhd a\neq0 $ for all non-zero $b\in V$, then there exists a nonzero element $a\lhd t$ or $t\lhd a\in V\cap I$ where $t=k_ma_m+\cdots+k_la_l\in V$. Denote this nonzero element by $w$.
For all $c\in V$,
\begin{equation*}
\begin{aligned}
&c_\lambda w=\partial(w\lhd c)+c\circ w+\lambda(w\lhd c),\\
&w_\lambda c=\partial(c\lhd w)+w\circ c+\lambda(c\lhd w).
\end{aligned}
\end{equation*}
Then one can get that $w\lhd c\in I$, $c\lhd w\in I$ by the coefficients of $\lambda$ and  $c\circ w\in I$, $w\circ c\in I$ by the coefficients of $\lambda^0$. Set $W=\{w\lhd c,c\lhd w,c\circ w,w\circ c| ~~\text{for all}~~a\in V \}$. Then $W\subseteq I$. For all
$d\in W$ and $b\in V$,
\begin{align*}
&b_\lambda d=\partial(d\lhd b)+b\circ d+\lambda(d\lhd b),\\
&d_\lambda b=\partial(b\lhd d)+d\circ b+\lambda(b\lhd d).
\end{align*}
Then we have $b\lhd d \in I$, $d\lhd b\in I$, $d\circ b\in I$ and $b\circ d\in I$. Therefore $V\subseteq I$ since $(V,\lhd,\circ)$ is simple. Consequently,  $R=\mathbb{C}\left[\partial\right]V=I$, i.e. $R$ is simple.
\end{proof}

\begin{rmk}
Note that there are some natural conditions to ensure that there exists an element $a\in V$ such that $a\lhd b\neq0$ or $b\lhd a\neq0 $ for all $b\in V$. For example,
 $(V,\lhd)$ has a right unit (resp. left unit) i.e. there exists an element $e\in V$ such that for all $a\in V$, $a\lhd e=a$ (resp. $e\lhd a=a$).
\end{rmk}

\begin{ex}\label{rank 2}
Let $(V=\mathbb{C}L\oplus \mathbb{C}W, \lhd, \rhd, \circ)$ be a two-dimensional pre-Gel'fand-Dorfman algebra with ``$\rhd$" trivial given as follows:
\begin{eqnarray*}
&&L\lhd L=L,\;\;W\lhd L=W,\;\; L\lhd W=W\lhd W=0,\\
&& L\circ L=0,\;\;L\circ W=W\circ L=h_1L,\;\;W\circ W=k_1(L+W),
\end{eqnarray*}
where $h_1,k_1\in\mathbb{C}\setminus\{0\}$. Note that $(V, \lhd, \rhd, \circ)$ is simple and $L$ is a right unit of $(V, \lhd)$.

Let $R=\mathbb{C}\left[\partial\right]V$ be the corresponding quadratic left-symmetric conformal algebra whose $\lambda$-products are given as follows:
\begin{align*}
&L_\lambda L=(\partial +\lambda)L,~~~~L_\lambda W=h_1L+(\partial +\lambda)W,\\
&W_\lambda L=h_1L,~~~~W_\lambda W=k_1L+k_1W,
\end{align*}
where $h_1,k_1\in\mathbb{C}\setminus\{0\}$. By Proposition \ref{simcor}, $R$ is simple.
\end{ex}

\begin{pro}\label{simcor1}
If a Novikov-Poisson algebra  $(V,\cdot,\circ )$ is simple,
 then the quadratic left-symmetric conformal algebra $R=\mathbb{C}\left[\partial\right]V $ corresponding to $(V,\cdot,\circ )$
is simple.
\end{pro}
\begin{proof}
With the assumption as the proof of Theorem \ref{simpleth} and by the coefficient of $\lambda^{j+1}$ in \eqref{lnds}, we get $w=b\cdot t\in I$ where $t=k_ma_m+\cdots+k_la_l$, for all $b\in V$.
We claim that there exists some nonzero element $w\in V\cap I$. If there exists some $b\in V$ such that $w=b\cdot t$ is not zero, then we are done. If $b\cdot t=0 $ for all $b\in V$, then we get that
$b\circ t+ c\cdot b\in I$ and $ t\circ b-b\cdot c\in I$ where $c=k_ga_g+\cdots+k_ha_h$ for any $b\in V$ by comparing coefficients of $\lambda^{j}$ in \eqref{lnds}.
If $b\circ t+ c\cdot b= t\circ b-b\cdot c=0$, we have
$$
b\circ t=- c\cdot b=-b\cdot c=-t\circ b
$$
for all $b\in V$. Let $W=t\circ V$. Note that for all $b$, $d\in V$,
\begin{align*}
&(t\circ b)\circ d=-(b\circ t)\circ d=-(b\circ d)\circ t\in W,\\
&d\circ(t\circ b)=(d\circ t)\circ b +t\circ (d\circ b)-(t\circ d)\circ b\in W,\\
&(t\circ b)\cdot d=(t\cdot d)\circ b=0=d\cdot (t\circ b).
\end{align*}
Therefore, $W$ is an ideal of $(V,\cdot,\circ)$. Thus $W=0$ or $W=V$. If $W=0$, then $\mathbb{C}t$ is an ideal of $(V,\cdot,\circ)$. If $\text{dim}V\geq 2$, it is impossible by the simplicity of $(V,\cdot,\circ)$.
When $\text{dim}V=1$, then $V=\mathbb{C}t$. Since $t\cdot t=t\circ t=0$, $V$ is trivial, which also contradicts with the simplicity of $(V,\cdot,\circ)$. Therefore, $W=V$.

On the other hand, by comparing the coefficient of $\lambda^{j-1}$ in the second equation of \eqref{lnds}, we obtain $\partial(b\cdot c)+e \in I$ for all $b\in V$ and some $e\in V$. Since for all $b$, $d\in V$,
\begin{align*}
(\partial (b\cdot c)+e)_\lambda d=&-\lambda(\partial(d\cdot (b\cdot c))+(b\cdot c)\circ d+\lambda(d\cdot(b\cdot c)))+\partial(d\cdot e)+e\circ d+\lambda(d\cdot e),
\end{align*}
we get that $d\cdot(b\cdot c)\in I$ by comparing the coefficient of $\lambda^{2}$. Since $V\cdot c=t\circ V=V$, there exist some $b,d\in V$ such that $d\cdot(b\cdot c)\neq 0$. Thus we get a nonzero element $w\in I\cap V$. For all $b\in V$,
\begin{align*}
&b_\lambda w=\partial(w\cdot b)+b\circ w+\lambda(w\cdot b),\\
&w_\lambda b=\partial(b\cdot w)+w\circ b+\lambda( b\cdot w).
\end{align*}
It is easy to find that $w\cdot b \in I$, $b\circ w\in I$ and $ w\circ b\in I$ for all $b\in V$ by comparing the coefficients of $\lambda$ and $\lambda^0$ respectively. Then $V\subseteq I$ due to that $(V, \cdot, \circ)$ is simple.  Since $I$ is a $ \mathbb{C}\left[\partial\right] $-module, one can get $I=R$. Thus $R$ is simple.
\end{proof}

\begin{cor}
Let $(V, \cdot, \circ)$ be a Novikov-Poisson algebra and $R=\mathbb{C}[\partial]V$ be the corresponding quadratic left-symmetric conformal algebra. Then $R=\mathbb{C}[\partial]V$ is simple if and only if $(V, \cdot, \circ)$ is simple.
\end{cor}
\begin{proof}
It is straightforward from Propositions \ref{simple1}  and \ref{simcor1}.
\end{proof}

Obviously, for a Novikov-Poisson algebra $(V, \cdot, \circ)$, if the Novikov algebra $(V, \circ)$ is simple, then $(V, \cdot, \circ)$ is simple. Note that the classification of Novikov-Poisson algebras on simple Novikov algebras with an idempotent element was given in \cite{X2}. Therefore, such classification result will provide many infinite simple left-symmetric conformal algebras. We present an example as follows.

\begin{ex}
Let $(V,\cdot,\circ )$ be the Novikov-Poisson algebra given in Example \ref{exp1}. Since $(V, \circ)$ is simple, $(V, \cdot, \circ)$ is simple.
Let $R=\mathbb{C}[\partial]V=\oplus_{i\in\mathbb{Z}}\mathbb{C}[\partial]x_i$ be the corresponding quadratic left-symmetric conformal algebra
given by
\begin{equation*}
{x_i}_\lambda x_j=(\partial +\lambda+j)x_{i+j},\;\;\text{for all}~~i,j\in\mathbb{Z}.
\end{equation*}
By Proposition \ref{simcor1}, $R$ is simple.
\end{ex}

\noindent {\bf  \small Acknowledgements}  \small The author would like to thank Professor Chengming Bai for fruitful discussions.

\noindent\noindent{\bf Declarations}

\noindent {\bf\small Ethical Approval} \small Not applicable.

\noindent{\bf \small Competing interests } \small  The authors declare that there is no conflict of interest.

\noindent{\bf \small Authors' contributions} \small Z. Xu  wrote Sections 3 and 4, Y. Hong wrote the introduction and Section 2. All authors reviewed the manuscript.

\noindent{\bf \small Funding} \small This research is supported by
NSFC (Nos. 12171129,
11871421) and the Zhejiang Provincial Natural Science
Foundation of China (No. LY20A010022) and the Scientific Research
Foundation of Hangzhou Normal University (No. 2019QDL012).

\noindent{\bf \small Availability of data and materials} \small Not applicable.

\end{document}